\documentclass[a4paper,notitlepage,twoside,reqno,12pt]{amsart}
\usepackage{CJK}
\usepackage{mathrsfs}
\usepackage[symbol]{footmisc}
\usepackage{bbm,pifont,latexsym}
\usepackage{anysize}  %
\marginsize{2.4cm}{2.4cm}{2.4cm}{2.4cm}
\usepackage{CJK}
\usepackage{dcolumn,indentfirst,color}
\usepackage[pdfpagemode=UseNone,pdfstartview=FitH]{hyperref}
\usepackage{amsmath,amssymb,amscd,amsthm,amsfonts,mathrsfs}
\usepackage{color,graphicx,xcolor,graphics,enumerate,comment}
\usepackage{titlesec} %
\usepackage{titletoc} %
\titlecontents{section}[20pt]{\addvspace{2pt}\filright}
              {\contentspush{\thecontentslabel. }}
              {}{\titlerule*[8pt]{}\contentspage}%

\usepackage{fancyhdr} %
\makeatletter
\@namedef{subjclassname@2020}{%
  \textup{2020} Mathematics Subject Classification}
\makeatother

\newtheorem{thm}{Theorem}[section]
\newtheorem{lem}[thm]{Lemma}
\newtheorem{cor}[thm]{Corollary}
\newtheorem{prop}[thm]{Proposition}
\newtheorem{rem}[thm]{Remark}

\newtheorem{defi}[thm]{Definition}

\theoremstyle{definition}

\newcommand{\bd}[1]{\mathbf{#1}}
\newcommand{\ud}{\mathrm{d}}
\newcommand{\norm}[1]{\left\lVert#1\right\rVert}
\newcommand{\N}{\mathbb{N}}
\newcommand{\T}{\mathbb{T}}

\newcommand{\R}{\mathbb{R}}
\newcommand{\Z}{\mathbb{Z}}
\newcommand{\Q}{\mathbb{Q}}

\newcommand{\MT}{\mathcal{T}}
\newcommand{\MM}{\mathcal{M}}

\newcommand{\MY}{\mathcal{Y}}

\newcommand{\MMM}{\mathfrak{M}}

\makeatletter\@addtoreset{equation}{section}\makeatother %

\titleformat{\section}{\centering\normalsize}{\textsc{\thesection.}}{0.5em}{\textsc}%
\titleformat{\subsection}[runin]{\normalsize}{\textbf{\thesubsection.}}{0.3em}{\textbf}

\begin{document}
\begin{CJK*}{GBK}{song}
\author{Wen-Long Li}
\address{Wen-Long Li, School of Sciences, Hangzhou Dianzi University, Hangzhou 310018, P. R. China}
\email{liwenlongchn@gmail.com}

\title{Variational construction of basic heteroclinic solutions for an Allen--Cahn equation}

\begin{abstract}
We establish a variational construction of minimal and without self-intersection solutions
for an Allen--Cahn equation,
especially for those corresponding to irrational rotation vectors.
These consequences generalize the results of rational cases of Rabinowitz and Stredulinsky 
to irrational cases.
\end{abstract}



\date{\today}

\maketitle



\section{Introduction}\label{sec:1}

In \cite{RS}, Rabinowitz and Stredulinsky studied the following Allen--Cahn equation:
\begin{equation}\label{eq:PDE}
  -\Delta u +F_u(x,u)=0, \quad x\in\R^n, \tag{PDE}
\end{equation}
where $u:\R^n\to \R$ and $F\in C^2(\T^{n+1})$, i.e.,
\begin{equation}\label{eq:F1-F2}
\left\{
  \begin{array}{l}
    F\in C^2 (\R^{n}\times \R),  \\
      \textrm{$F$ is $1$-periodic in $x_1,\cdots,x_n$ and in $u$.}
  \end{array}
\right.
\end{equation}
The study on equation like \eqref{eq:PDE} was first initialed by Moser \cite{moser},
who wanted to obtain a PDE version of Aubry--Mather theory.
Among other things, 
Aubry--Mather theory studies minimal solutions of some $1$ dimensional variational problems.
For \eqref{eq:PDE},
a function $u\in W^{1,2}_{loc}(\R^n)$ is said to be \emph{minimal} 
if for any $\phi\in W^{1,2}_{loc}(\R^n)$ with compact support, 
it holds 
\begin{equation*}
  \int_{\R^n}[L(u+\phi)-L(u)]\ud x\geq 0,
\end{equation*}
where
\begin{equation*}
  L(u):=\frac{1}{2}|\nabla u|^2 +F(x,u).
\end{equation*}
Moser noted that different from Aubry--Mather theory, 
where a minimal solution does not cross with its $\Z^2$ translations,
for higher dimensional spaces, this property fails.
So Moser asked functions to satisfy the property of without self-intersection (WSI for short),
that is
for $u\in C^0(\R^n)$, $u$
satisfies
\begin{equation*}
u> \MT_{\bar{\bd{k}}}u, \quad \textrm{ or } \quad u= \MT_{\bar{\bd{k}}}u, \quad\textrm{ or } \quad
 u< \MT_{\bar{\bd{k}}}u,
\end{equation*}
for any
$\bar{\bd{k}}=(\bd{k},\bar{\bd{k}}_{n+1})\in\Z^{n}\times \Z$ with
$\MT_{\bar{\bd{k}}}u(\cdot):=u(\cdot-\bd{k})+\bar{\bd{k}}_{n+1}$.
In this paper,
we shall denote the elements of $\Z^k$ with $2\leq k\leq n$ 
by bold $\bd{i}, \bd{j}, \bd{k}, \bd{p}, \bd{q}$, etc.,
and the components of $\bd{i}$ by $\bd{i}_1, \cdots, \bd{i}_k$.
Elements of $\Z^{n+1}$ will be denoted by, e.g., $\bar{\bd{i}}=(\bd{i}, \bar{\bd{i}}_{n+1})\in \Z^n\times \Z$.

Similar to Aubry--Mather theory,
Moser defined a rotation vector $\alpha=\alpha(u)\in\R^n$
for any minimal and WSI solution $u$.
\begin{lem}[{cf. \cite[Theorem 1.2]{RS}}]\label{lem:mm1.2}
If $F$ satisfies \eqref{eq:F1-F2}
and $u$ is a solution of \eqref{eq:PDE}
that is minimal and WSI,
then there is an $\alpha=\alpha(u)\in\R^n$ such that
\begin{equation*}
  |u(x)-\alpha\cdot x|
\end{equation*}
is bounded on $\R^n$.
\end{lem}
Denote by $\mathfrak{M}_{\alpha}$ the set of all minimal and WSI solutions with rotation vector $\alpha$.
Moser showed for any $\alpha\in\R^n$, $\mathfrak{M}_{\alpha}\neq \emptyset$.
For $\alpha\in\Q^n$, this is proved by minimization method.
Other cases are shown by an approximation argument.

Later Bangert \cite{Bangert} studied Moser's problem 
and made a classification for minimal and WSI solutions.
He introduced the so-called second invariants for this purpose.
For a WSI function $u\in C^0 (\R^n)$,
using the property of WSI and the property of $\Z^n$,
Bangert defined an integer $t=t(u)\in\N$ and unit vectors
$\bar{a}^1(u),\cdots, \bar{a}^t(u)\in\R^{n+1}$ satisfying
\begin{equation}\label{eq:lkdkl22d}
\left\{
  \begin{array}{l}
    \bar{a}^1(u)\cdot \bar{\bd{e}}^{n+1}>0,  \\
    \bar{a}^s(u)\in span(\bar{\Lambda}_s) \textrm{ for $2\leq s\leq t$},
  \end{array}
\right.
\end{equation}
where
$\bar{\bd{e}}^{n+1}=(0,\cdots,0,1)\in\R^{n+1}$
and
$$\bar{\Lambda}_s:=\bar{\Lambda}_s(\bar{a}^1(u),\cdots, \bar{a}^{s-1}(u)):=\Z^{n+1}\cap \langle \bar{a}^1(u), \cdots, \bar{a}^{s-1}(u) \rangle ^{\perp}$$
for $2\leq s\leq t+1$.
Here $\langle \bar{a}^1(u), \cdots, \bar{a}^{s-1}(u) \rangle ^{\perp}\subset \R^{n+1}$
is the orthogonal complement of the linear subspace generated by $\bar{a}^1(u), \cdots, \bar{a}^{s-1}(u)$.
The integer $t(u)$ and vectors $\bar{a}^1(u),\cdots, \bar{a}^t(u)$ are called by Bangert the second invariants of $u$.
Set $\bar{\Lambda}_1:=\Z^{n+1}$,
then $t(u), \bar{a}^1(u),\cdots, \bar{a}^t (u)$ are characterized by the following lemma.
\begin{lem}[{cf. \cite[(3.6)-(3.7)]{Bangert}}]\label{prop:dk68956}
\
\begin{enumerate}
  \item $\MT_{\bar{\bd{k}}}u>u$
  if and only if there exists $1\leq s\leq t$ such that $\bar{\bd{k}}\in \bar{\Lambda}_s$ and $\bar{\bd{k}}\cdot \bar{a}^s>0$.
  \item $\MT_{\bar{\bd{k}}}u=u$
  if and only if $\bar{\bd{k}}\in\bar{\Lambda}_{t+1}$.
\end{enumerate}
\end{lem}

Denote all the minimal and WSI solutions 
with the same second invariants $t$ and $\bar{a}^1,\cdots, \bar{a}^t$ by
$\mathfrak{M}(\bar{a}^1,\cdots, \bar{a}^t)$.
For a minimal and WSI solution $u$,
Moser's rotation vector $\alpha=\alpha(u)$ and
Bangert's invariant $\bar{a}^1=\bar{a}^1(u)$ satisfy the following relation
$$\bar{a}^1=\frac{(-\alpha,1)}{(\alpha_1 ^2+\cdots \alpha_n ^2 +1)^{\frac{1}{2}}}.$$
If $t\geq 2$, we shall call the elements in $\mathfrak{M}(\bar{a}^1,\cdots, \bar{a}^t)$ 
basic heteroclinic solutions.
Motivated by the above facts, Bangert introduced the following definition.
\begin{defi}[cf. {\cite[Definition 3.9]{Bangert}}]\label{def:admissible}
A system of $(\bar{a}^1,\cdots, \bar{a}^t)$ of unit vectors in $\R^{n+1}$ is admissible
if $\bar{a}^1,\cdots, \bar{a}^t\in\R^{n+1}$ satisfy \eqref{eq:lkdkl22d}.
\end{defi}
It is obvious that if $(\bar{a}^1,\cdots, \bar{a}^t)$ is admissible, then for all $1\leq s\leq t$,
$(\bar{a}^1,\cdots, \bar{a}^s)$ is admissible.
Assuming some gap conditions,
Bangert showed that for any admissible system $\bar{a}^1,\cdots, \bar{a}^t$,
$\mathfrak{M}(\bar{a}^1,\cdots, \bar{a}^t)\neq \emptyset$.
Moreover, if $\bar{a}^1,\cdots, \bar{a}^t$ is admissible,
then
\begin{equation}\label{eq:ldffdfddf}
  \mathfrak{M}(\bar{a}^1)\cup \mathfrak{M}(\bar{a}^1, \bar{a}^2) \cup
  \cdots\cup \mathfrak{M}(\bar{a}^1,\cdots, \bar{a}^t)(\subset \mathfrak{M}_{\alpha}) \textrm{ is ordered.}
\end{equation}
Bangert's proofs of these results depend on careful studies of the structure of
$\mathfrak{M}(\bar{a}^1,\cdots, \bar{a}^t)$,
in which ordered relation plays an important role.

For the rational rotation vector $\alpha\in\Q^n$,
Rabinowitz and Stredulinsky \cite{RS} developed a minimization method and gave the variational
construction for Bangert's basic heteroclinic solutions.
Then by this variational structure and
using these minimal and WSI solutions as building blocks,
they constructed a large family of complicated homoclinic and heterocolinic solutions for \eqref{eq:PDE}.
So naturally we ask:
can one generalize Rabinowitz and Stredulinsky's results to the case of $\alpha\in\R^n\setminus \Q^n$?

The answer is affirmative and it is the aim of the present paper.
More precisely, in this paper
we shall use variational method to give a variational construction for Bangert's basic heteroclinic solutions,
especially for those 
corresponding to irrational rotational vectors.
Then Rabinowitz and Stredulinsky's arguments can be applied 
and more solutions of \eqref{eq:PDE} will be obtained 
(cf. \cite{RS}, see also \cite{Rabi2006, Rabi2007, Rabi2011, Rabi2014}).
The idea of our method is similar to that of Rabinowitz and Stredulinsky,
and is also inspired by \cite{Mather, dev2007}.
The key point is 
we observe that the functional $G_{\omega}$ (cf. \cite[p. 190]{Mather}, see also \cite[(28)]{dev2007})  
can be used to generalize the renormalized functionals of Rabinowitz and Stredulinsky.

In \cite[Section 7]{moser}, 
Moser proposed an alternate variational principle to obtain minimal and WSI solutions of
\eqref{eq:PDE}.
This approach is based on Mather's method (\cite{Mather1985}).
Bessi (\cite{bessi}) completed this construction.
But it seems not clear how to establish 
more complicated solutions of Rabinowitz and Stredulinsky 
by Bessi's result.


\subsection*{Sketch of the results}
\

Let us briefly state our results of the present paper.
The main results will be stated in Section \ref{sec:3}.
The main results concern the variational construction of Bangert's basic heteroclinic solutions.
Unlike Bangert's method,
our approach is a combination of the methods of Moser (\cite{moser}) and Rabinowitz--Stredulinsky (\cite{RS}),
and we shall ``build-up'' all the minimal and WSI solutions of \eqref{eq:PDE}.
We outline our approach and construction here. 

\noindent
\textbf{Step 1:} the construction of $\MMM(\bar{a}_1)$
\begin{enumerate}[(I)]
\item If $\alpha\in\Q^n$,
by Moser's minimization method,
one have a periodic function $u\in C^2(\R^n)$ such that
$u+\alpha\cdot x$ is a solution of \eqref{eq:PDE}.
In Bangert's notation,
$u+\alpha\cdot x\in\MMM(\bar{a}^1)$ with
$\bar{a}^1=(-\alpha,1)/\norm{(-\alpha,1)}$ and $\norm{(-\alpha,1)}:=(\alpha_1 ^2 + \cdots \alpha_n ^2 +1)^{1/2}$.
Moreover, if $\alpha\in\Q^n$, 
all the elements of $\MMM(\bar{a}^1)$ are periodic and are minimizers of the corresponding functional.
See \cite[Theorem 3.60]{RS} for the details.
\item \label{step1-II}
If $\alpha\not\in\Q^n$,
Moser showed there is some $u\in\MMM(\bar{a}^1)$ by approximation method, and $u$ can be chosen such that $\MT_{\bar{\bd{k}}}u=u$ for all $\bar{\bd{k}}\in \bar{\Lambda}_2:=\bar{\Lambda}_2(\bar{a}^1):=\Z^{n+1}\cap \langle \bar{a}^1 \rangle ^{\perp}$.
Define
$$\MMM_1(u):=\textrm{closure of }\{\MT_{\bar{\bd{k}}}u\,|\, \bar{\bd{k}}\in \bar{\Lambda}_1\},$$
where the closure is taken with respect to $C^1$-convergence on compact sets.
Then there is a unique minimal set, denoted by $\MMM^{rec}(\bar{a}^1)$, contained in $\MMM_1(u)$.
Bangert's uniqueness theorem (\cite{Bangert}) tells us that
$\MMM^{rec}(\bar{a}^1)$ is independent of $u$.
In fact,
\begin{equation*}
  \MMM^{rec}(\bar{a}^1)
  =\{u\in \MMM(\bar{a}^1)\,|\, u=\sup_{\bar{\bd{k}}\in\bar{\Lambda}_1 \atop \bar{\bd{k}}\cdot \bar{a}^1<0}\MT_{\bar{\bd{k}}}u \,\,\textrm{ or }\,\, u=\inf_{\bar{\bd{k}}\in\bar{\Lambda}_1 \atop \bar{\bd{k}}\cdot \bar{a}^1>0}\MT_{\bar{\bd{k}}}u\}.
\end{equation*}

Two cases may occur:
either (a) $\{u(\bd{0})\,|\, u\in\MMM^{rec}(\bar{a}^1)\}=\R$
or (b) $\{u(\bd{0})\,|\, u\in\MMM^{rec}(\bar{a}^1)\}\subset\R$ is a Cantor set.
If (a) (resp. (b)) holds, we say $\MMM^{rec}(\bar{a}^1)$ is a foliation (resp. lamination).
If (a) occurs, $\MMM^{rec}(\bar{a}^1)=\MMM(\bar{a}^1)$.
Then there does not exist basic heteroclinic solution with invariants $t\geq 2$
and $\bar{a}^1, \cdots, \bar{a}^t$ for any admissible $(\bar{a}^1, \cdots, \bar{a}^t)$.
Next we focus on case (b).

In case (b),
as Bangert pointed out at the end of Section 9 in \cite{dynarep},
(for $n=1$) 
there may be some element in $\MMM(\bar{a}^1)$
which is non-recurrent.
To obtain all the minimal and WSI solutions,
we need to construct the elements in $\MMM(\bar{a}^1)\setminus\MMM^{rec}(\bar{a}^1)$.
Note that $\MMM^{rec}(\bar{a}^1)$ is an ordered set.
Now assume the following gap condition:
\begin{equation}\label{eq:*00}
\textrm{$v_0<w_0\in \mathfrak{M}^{rec}(\bar{a}^1)$ are an adjacent pair}. \tag{$*_0$}
\end{equation}
Here $v_0<w_0$ are adjacent in $\MMM^{rec}(\bar{a}^1)$ means there
does not exist $u\in\MMM^{rec}(\bar{a}^1)\setminus\{v_0, w_0\}$ such that $v_0\leq u\leq w_0$.
With the gap condition \eqref{eq:*00},
we can construct $u\in\MMM(\bar{a}^1)\setminus\MMM^{rec}(\bar{a}^1)$ 
lying between $v_0$, $w_0$ as follows.
For $u\in \Gamma_1(v_0, w_0):=\{\bar{u}\in W^{1,2}_{loc}(\R^n)\,|\, v_0\leq \bar{u}\leq w_0, \textrm{ and } \MT_{\bar{\bd{k}}}\bar{u}=\bar{u} \textrm{ for } \bar{\bd{k}}\in \bar{\Lambda}_2\}$,
set $$J_1(u)=\int_{\tilde{\T}}[L(u)-L(v_0)]\ud x,$$
where $\tilde{\T}$ is a fundamental domain of $\R^{n}/\Lambda_2$, and
$\Lambda_2:=\pi(\bar{\Lambda}_2)$ with $\pi : \R^{n+1}\to \R^n$ is defined by $\pi((x,x_{n+1}))=x$.
It can be proved (see Section \ref{sec:3}) that
$\min_{u\in \Gamma_1(v_0, w_0)}J_1(u)=0$ is attained and
$$\{u\in \Gamma_1(v_0, w_0)\,|\, J_1(u)=0\}=\MMM(\bar{a}^1)\cap \Gamma_1(v_0, w_0).$$
Thus the elements of $\MMM(\bar{a}^1)\cap \Gamma_1(v_0, w_0)$
have a variational construction.

Now for any fixed $u\in\MMM(\bar{a}^1)$, 
does it lie in some adjacent pair of $\MMM^{rec}(\bar{a}^1)$ provided $u\not\in\MMM^{rec}(\bar{a}^1)$? Yes!
Since $u\in\MMM(\bar{a}^1)$ implies $u$ is WSI,
we have $\MMM_1(u)$ is ordered.
Thus $\MMM^{rec}(\bar{a}^1)\cup \{u\}$ is also ordered since $\MMM^{rec}(\bar{a}^1)\cup \{u\}\subset \MMM_1(u)$.
Hence either $u\in \MMM^{rec}(\bar{a}^1)$ or there exist
adjacent pair $v_0<w_0\in \MMM^{rec}(\bar{a}^1)$,
such that
$u$ lies in $\Gamma_1(v_0,w_0)$.
\end{enumerate}

\begin{rem}
Note that we have used Bangert's uniqueness theorem (i.e., $\MMM^{rec}(\bar{a}^1)$ 
is independent of $u\in\MMM(\bar{a}^1)$),
and we do not need ``$\MMM(\bar{a}^1)$ is ordered''.
In fact, our construction provides another proof for the fact ``$\MMM(\bar{a}^1)$ is ordered''.
\end{rem}

\noindent
\textbf{Step 2:} the construction of $\MMM(\bar{a}^1,\bar{a}^2)$

Assuming $\bar{a}^1, \bar{a}^2$ are admissible,
let us consider the basic heteroclinic solutions in $\MMM(\bar{a}^1,\bar{a}^2)$.
Note here we do not require $\alpha\in\Q^n$.

If $rank(\bar{\Lambda}_3)=rank(\bar{\Lambda}_2)-1$,
using the variational construction in Step 1,
Rabinowitz and Stredulinsky's method can be generalized to this case and obtain basic heteroclinic solutions.
The key point is that the definition of the second functional only depends on the first functional.

If $rank(\bar{\Lambda}_3)<rank(\bar{\Lambda}_2)-1$,
we can repeat the method of step 1 \eqref{step1-II} to obtain $\MMM^{rec}(\bar{a}^1, \bar{a}^2)$.
Then using minimization method we get elements in $\MMM(\bar{a}^1, \bar{a}^2)\setminus\MMM^{rec}(\bar{a}^1, \bar{a}^2)$.

The above program can be repeated to obtain all the basic heteroclinic solutions of \eqref{eq:PDE}.

\bigskip

The rest of the paper is organized as follows.
In Section \ref{sec:2} some preliminaries are provided.
In Section \ref{sec:3},
we define the first functional whose minimizers are elements in $\MMM(\bar{a}^1)$.
In Section \ref{sec:4} basic heteroclinic solutions in $\MMM(\bar{a}^1,\bar{a}^2)$ are constructed by minimization method.


\section{Preliminary}\label{sec:2}

We recall some results of Bangert \cite{Bangert, bangert-uniq} and Rabinowitz--Stredulinsky \cite{RS}, which are fundamental for our analysis.

Assume $(\bar{a}^1,\cdots, \bar{a}^t)$ are admissible.
Recall that $\bar{\Lambda}_1:=\Z^{n+1}$ and
$\bar{\Lambda}_s$ ($2\leq s\leq t+1$) are defined as in \eqref{eq:lkdkl22d}, i.e.,
$\bar{\Lambda}_s:=\bar{\Lambda}_s(\bar{a}^1,\cdots, \bar{a}^s):=\Z^{n+1}\cap \langle \bar{a}^1, \cdots, \bar{a}^{s-1} \rangle ^{\perp}$.
For $1\leq s \leq t$,
set
\begin{equation}\label{eq:133434134}
  \MMM^{rec}(\bar{a}^1,\cdots,\bar{a}^s)
  :=\{u\in \MMM(\bar{a}^1, \cdots,\bar{a}^s)\,|\, u=\sup_{\bar{\bd{k}}\in\bar{\Lambda}_s \atop \bar{\bd{k}}\cdot \bar{a}^s<0}\MT_{\bar{\bd{k}}}u \,\,\textrm{ or }\,\, u=\inf_{\bar{\bd{k}}\in\bar{\Lambda}_s \atop \bar{\bd{k}}\cdot \bar{a}^s>0}\MT_{\bar{\bd{k}}}u\}.
\end{equation}
The remark following \cite[Definition 5.2]{Bangert} shows that 
if $rank(\bar{\Lambda}_{s+1})= rank(\bar{\Lambda}_{s})-1$, 
then $\MMM^{rec}(\bar{a}^1,\cdots,\bar{a}^s)=\emptyset$.
For $1\leq s \leq t$,
assume $u\in \MMM(\bar{a}^1,\cdots,\bar{a}^{t})$.
Define
$$\MMM_t(u):=\textrm{closure of }\{\MT_{\bar{\bd{k}}}u\,|\, \bar{\bd{k}}\in \bar{\Lambda}_t\},$$
where the closure is taken with respect to $C^1$-convergence on compact sets.
For any functions $v<w$, denote by $(v,w)$ as the set
$\{(x,x_{n+1})\,|\, v(x)< x_{n+1}< w(x)\}$.
Bangert showed the following deep results.

\begin{lem}[{cf. \cite[Theorem 5.1, (3.8)-(3.9)]{bangert-uniq}, \cite[Lemma 6.14, Proposition 5.5, Lemma 5.1]{Bangert}}]\label{lem:659712}
\
\begin{enumerate}
  \item If $0\leq rank(\bar{\Lambda}_{2})\leq rank(\bar{\Lambda}_{1})-2$, $\MMM^{rec}(\bar{a}^1)$ is totally ordered. The set $\{u(\bd{0})\,|\, u\in\MMM^{rec}(\bar{a}^1)\}$ is either a Cantor set, or $\R$.
  \item Suppose $\MMM(\bar{a}^1,\cdots,\bar{a}^{t-1})$ is totally ordered and $t>1$. Then $\MMM^{rec}(\bar{a}^1,\cdots,\bar{a}^t)$ is totally ordered. And for any adjacent $v<w\in \MMM(\bar{a}^1,\cdots,\bar{a}^{t-1})$,
      either $\{u(\bd{0})\,|\, u\in\MMM^{rec}(\bar{a}^1,\cdots,\bar{a}^t)\cap (v,w)\}$ is a Cantor set in $(v(\bd{0}),w(\bd{0}))$, or is equal to $(v(\bd{0}),w(\bd{0}))$.
      If $ rank(\bar{\Lambda}_{t+1})= rank(\bar{\Lambda}_{t})-1$, 
      $\MMM_t(u)\cap (v,w)$ is discrete in $(v,w)$ for any $u\in \MMM(\bar{a}^1,\cdots,\bar{a}^t)$.
\end{enumerate}
\end{lem}
\begin{lem}[cf. {\cite[Corollary 6.21]{Bangert}}]\label{lem:dlffdf}
For $1\leq s \leq t$ and any $u\in \MMM(\bar{a}^1,\cdots,\bar{a}^{t})$,
$$\MMM^{rec}(\bar{a}^1,\cdots,\bar{a}^t)\subset \MMM_t(u).$$
Since $u$ is WSI, $\MMM^{rec}(\bar{a}^1,\cdots,\bar{a}^t)\cup \{u\}$ is ordered.
\end{lem}

Although Bangert showed \eqref{eq:ldffdfddf}, a deeper consequence,
we shall use Lemmas \ref{lem:659712}-\ref{lem:dlffdf} to give the variational construction.

To introduce the following lemmas, some notations are needed.
Denote by $\pi: \R^{n+1}\to \R^{n}$ the projection of $\pi(x,x_{n+1})=x\in\R^n$.
Set $\Lambda_s:=\pi(\bar{\Lambda}_s)$, $\bar{V}_s=span(\bar{\Lambda}_s)$,
and
$V_{s}:=\pi(\bar{V}_s)=span(\Lambda_s)$ for $1\leq s\leq t$.
Abusing notations,
let $V_{s}/\Lambda_s$ be a measurable fundamental domain of $V_{s}/\Lambda_s$ and $\pi_s:\R^n\to V_{s}$.
Assume $V_{t}^{\perp}$ is the orthogonal space of $V_{t}$ in $\R^n$.
\begin{lem}[cf. {\cite[Lemma 4.5]{Bangert}}]\label{lem:bangert11}
Suppose that $v,w\in C^0 (\R^n)$
satisfy
\begin{equation}\label{eq:ddlpp}
  \textrm{if $1\leq s<t$ and $\bar{\bd{k}}\in\bar{\Lambda}_s$ and $\bar{\bd{k}}\cdot \bar{a}^s>0$, then
$\MT_{\bar{\bd{k}}}v\geq w$.}
\end{equation}
Then
\begin{equation}\label{eq:lglf}
  \int_{V_{t}^{\perp}\times (V_t/\Lambda_t)}(w-v)\ud x \leq 1,
\end{equation}
and
\begin{equation}\label{eq:lglddf}
  \int_{V_{t}^{\perp}\times (V_t/\Lambda_t)}(w-v)^2\ud x \leq 1.
\end{equation}
\end{lem}

\begin{rem}[{cf. \cite[p. 108, line 1-2]{Bangert}}]\label{rem:lsdjfld}
If $v,w\in\MMM^{rec}(\bar{a}^1,\cdots, \bar{a}^{t-1})$
are adjacent, then $v,w$ satisfy \eqref{eq:ddlpp}.
This is the most important case.
For the generic case, we have $$\MMM(\bar{a}^1,\cdots, \bar{a}^{t-1})=\MMM^{rec}(\bar{a}^1,\cdots, \bar{a}^{t-1}),$$
and in this case \eqref{eq:ddlpp} holds only for adjacent pair in $\MMM(\bar{a}^1,\cdots, \bar{a}^{t-1})$.
\end{rem}

\begin{lem}\label{rem:kdkdk}
If $\alpha\in\Q^n$, then for any $v,w\in \MMM(\bar{a}^1,\cdots, \bar{a}^{t-1})$,
\begin{equation*}
  \int_{V_{t}^{\perp}\times (V_t/\Lambda_t)}(w-v)\ud x \leq C
\end{equation*}
for some $C=C(\alpha; v,w)$.
\end{lem}

\begin{rem}
Lemma \ref{rem:kdkdk} is a variant and a generalization of Lemma \ref{lem:bangert11}.
Rabinowitz and Stredulinsky used this fact repeatedly in constructing renormalized functionals.

But for $\alpha\not\in\Q^n$, Lemma \ref{lem:bangert11} cannot be generalized any more.
One reason is that for $\alpha\not\in\Q^n$, $V_{t}^{\perp}\times (V_t/\Lambda_t)$ is an unbounded domain for $t\geq 1$.
Note that in \cite{RS}, the definition of $J_2$ is related to an unbounded domain but its domain lies between $v_0,w_0$, which are adjacent in  $\MMM(\bar{a}^1)$.
In other words, Rabinowitz and Stredulinsky's construction begins from a bounded domain while our construction
from an unbounded domain.
As we shall see in the following sections, our renormalized functional is a conjunction of $J_1$ and $J_2$ of Rabinowitz and Stredulinsky.
\end{rem}

\begin{lem}[{cf. \cite[Lemma 3.62]{RS}}]\label{lem:3.6211}
Suppose $1\leq\dim(V_t)\leq n-1$.
If $u \in W_{loc}^{1,2}\left(V_{t}^{\perp}\times (V_t/\Lambda_t)\right)$ is minimal, then for any
$\phi\in W_{loc}^{1,2}\left(V_{t}^{\perp}\times (V_t/\Lambda_t)\right)$ with compact support,
\begin{equation}\label{eq:3.63-1}
\int_{V_{t}^{\perp}\times (V_t/\Lambda_t)}[L(u+\phi)-L(u)] \ud x \geq 0.
\end{equation}
\end{lem}

\begin{proof}
Since $V_t$ is generated by elements of $\pi(\Z^{n+1})$,
we can choose some elements, say $\bd{b}^{i+1},\cdots, \bd{b}^n\in\Z^n\cap V_t$ as basis of $V_t$, where $n-i=\dim(V_t)\in [1, n-1]$ by assumption.
We can also choose some elements $\bd{b}^{1},\cdots, \bd{b}^{i}\in\Z^n\cap V_{t}^{\perp}$ as basis of $V_{t}^{\perp}$.

Denote by $y=(\tilde{y},\hat{y})\in V_t ^{\perp} \times V_t $ the coordinates of
the point $\tilde{y}_1 \bd{b}^{1}+ \cdots +\tilde{y}_i \bd{b}^{i}+ \hat{y}_{i+1}\bd{b}^{i+1} +\cdots + \hat{y}_n \bd{b}^{n}$.
Let
$j\in\N$.
For $s\in\R$, let $\theta_j (s)\in C^1(\R,\R)$ such that
\begin{equation*}
  \theta_{j}(s)=\left\{
                  \begin{array}{ll}
                    1, & |s|\leq j, \\
                    0, & |s|\geq j+1,
                  \end{array}
                \right.
\end{equation*}
and $0\leq \theta_j (s)\leq 1$ with $|\theta ' _{j}(s)|\leq 2$.

Suppose the support of $\phi$ lies in $[\bd{p}_1,\bd{q}_1]\bd{b}^1\times \cdots \times [\bd{p}_{i},\bd{q}_{i}]\bd{b}^{i}  \times (V_t/\Lambda_t)$ with $\bd{p}_j,\bd{q}_j\in\Z$.
Let $y=(\bd{b}^1,\cdots, \bd{b}^n)x$
and $D=\left| \det \frac{\partial (x_1, \cdots, x_n)}{\partial (y_1, \cdots, y_n)} \right|$,
where $\det(A)$ is the determinant of matrix $A$.
Since $u$ is minimal,
\begin{equation}\label{eq:3.65}
\begin{aligned}
0 &\leq \int_{\mathbb{R}^{n}}\left[L\left(u+\theta_{j}(|\hat{y}|) \phi\right)-L(u)\right] 
D \ud y\\
 &\leq \int_{\left(\prod_{m=1}^{i}[\bd{p}_m,\bd{q}_m]\right) \times[-j-1, j+1]^{n-i}}
 \left[L\left(u+\theta_{j} \phi\right)-L(u)\right] D \ud y \\
&=\int_{\left(\prod_{m=1}^{i}[\bd{p}_m,\bd{q}_m]\right) \times[-j, j]^{n-i} }[L(u+\phi)-L(u)] 
D \ud y +\mathcal{R}_{j}(u, \phi)\\
&=(2j)^{n-i} \int_{\left(\prod_{m=1}^{i}[\bd{p}_m,\bd{q}_m]\right) \times[0,1]^{n-i} }[L(u+\phi)-L(u)] 
D \ud y  \\
 &\quad \quad
 +\mathcal{R}_{j}(u, \phi),
\end{aligned}
\end{equation}
where
\begin{equation*}
\mathcal{R}_{j}(u, \phi)=\int_{B_{j}}\left[L\left(u+\theta_{j} \phi\right)-L(u)\right]
 D \ud y ,
\end{equation*}
and
\begin{equation*}
B_{j}=\left(\prod_{m=1}^{i}[\bd{p}_m,\bd{q}_m]\right) \times\left([-j-1, j+1]^{n-i} 
\backslash[-j, j]^{n-i}\right).
\end{equation*}
Noticing $u,\phi\in W_{loc}^{1,2}\left(V_{t}^{\perp}\times (V_t/\Lambda_t)\right)$, 
for the $(\bd{q}_1-\bd{p}_1)\cdots (\bd{q}_{i}-\bd{p}_{i})\left[(2(j+1))^{n-i}-(2 j)^{n-i}\right]$
unit cubes $b_{\ell} \subset \R^n$ that make up $B_{j}$,
we have 
\begin{equation*}
  \left| \int_{b_{\ell}}\left[L\left(u+\theta_{j} \phi\right)-L(u)\right] 
  D \ud y \right| \leq M,
\end{equation*}
where $M:=M(u,\phi)$ does not depend on $j$ by the definition of $\theta_j$.
Thus
\begin{equation}\label{eq:3.66}
  \mathcal{R}_{j}(u, \phi)\leq M b j^{n-i-1},
\end{equation}
where $b$ depends on $\bd{p}$, $\bd{q}$, $n$ and $\phi$.
Therefore by \eqref{eq:3.65}-\eqref{eq:3.66},
\begin{equation}\label{eq:3.67}
\begin{split}
0 \leq&\left(2j\right)^{n-i}  \int_{\left(\prod_{m=1}^{i}[\bd{p}_m,\bd{q}_m]\right) \times[0,1]^{n-i}}
[L(u+\phi)-L(u)] D \ud y \\
&\quad
+M b j^{n-i-1}.
\end{split}
\end{equation}
Thus dividing \eqref{eq:3.67} by $(2j)^{n-i}$ and letting $j\to\infty$,
we obtain \eqref{eq:3.63-1}.
\end{proof}

\begin{lem}[cf. {\cite[Corollary 2.8, Lemma 3.10]{Bangert}}]\label{lem:dldqwloplo}
\begin{enumerate}
  \item Every sequence $u_k$ with $u_k\in\MMM_{\alpha_k}$ and both $|u_k(\bd{0})|$ and $|\alpha_k|$ bounded contains a subsequence which is $C^1$-convergent on compact sets to some minimal and WSI solution $u$.
  \item Assume that $u, u_k\in C^0(\R^n)$ are WSI
and $u_k\to u$ pointwise as $k\to \infty$.
Then $\bar{a}^1(u_k)\to \bar{a}^1(u)$ as $k\to \infty$.
\end{enumerate}
\end{lem}

\begin{lem}[cf. {\cite[Lemma 3.11]{Bangert}}]\label{lem:dldloplrrtyto}
Assume that $u$ is WSI.
Suppose for some $s\leq t(u)$ there exists
a unit vector $\bar{a}\in span (\bar{\Lambda}_s(u))$ such that $\bar{\bd{k}}\in\bar{\Lambda}_s(u)$ and
$\bar{\bd{k}}\cdot \bar{a}>0$ imply $\MT_{\bar{\bd{k}}}u\geq u$.
Then $\bar{a}=\bar{a}^s(u)$.
\end{lem}

\begin{lem}[{cf. \cite[(p. 10, line -7)-(p. 11, line 3)]{RS}, \cite[Section 4]{moser}, \cite[Lemma 2.2]{Bangert}}]\label{lem:rab}
Suppose $u,v$ are classical solutions of \eqref{eq:PDE}.
If $u\leq v$ on a connect open set $\Omega\subset \R^n$,
then either $u<v$ or $u\equiv v$ on $\Omega$.
\end{lem}

\begin{prop}
Assume that $u$ is minimal.
If
$0\neq\phi\in W^{1,2}_{loc}(\R^n)$ has compact support and $\phi\geq 0$ (or $\phi\leq 0$),
then
\begin{equation*}
  \int_{\R^n}[L(u+\phi)-L(u)]\ud x>0.
\end{equation*}
\end{prop}

\begin{proof}
By the definition of minimal,
\begin{equation}\label{eq:dlkjflgggdfd}
  \int_{\R^n}[L(u+\phi)-L(u)]\ud x\geq 0.
\end{equation}
Suppose by contradiction it holds equality in \eqref{eq:dlkjflgggdfd}.
Then for any $\psi\in W^{1,2}_{loc}(\R^n)$ with compact support,
\begin{equation}\label{eq:dlkjflggddgdfd}
\begin{split}
   &\int_{\R^n}[L(u+\phi+\psi)-L(u+\phi)]\ud x\\
   =&\int_{\R^n}\{[L(u+\phi+\psi)-L(u)]-[L(u+\phi)-L(u)]\}\ud x\\
   =&\int_{\R^n}[L(u+\phi+\psi)-L(u)]\ud x\\
   \geq &0.
   \end{split}
\end{equation}
So $u+\phi$ is also minimal and thus is a classical solution of \eqref{eq:PDE},
contrary to Lemma \ref{lem:rab}.
\end{proof}


\section{Variational construction of $\MMM(\bar{a}^1)$}\label{sec:3}


\subsection{The case $rank(\bar{\Lambda}_2)=rank(\bar{\Lambda}_1)-1$}
\

In this case,
$\alpha\in\Q^{n}$ and
$\MMM(\bar{a}^1)$ consists of periodic solutions.
This case has been treated by Moser \cite{moser}.
See also Rabinowitz--Stredulinsky \cite{RS}.


\subsection{The case $1\leq rank(\bar{\Lambda}_2)\leq rank(\bar{\Lambda}_1)-2$}\label{sec:3.2}
\

In this case,
assume $rank(\bar{\Lambda}_2)=n-n_2\in [1,n-1]$ with $1\leq n_2\leq n-1$.
Since $rank(\Lambda_2)=rank(\bar{\Lambda}_2)$,
$rank(\Lambda_2)=n-n_2$.
The typical and simplest example is 
$$\bar{a}^1=\frac{(-\alpha,1)}{\|(-\alpha,1)\|}
=\frac{(-\alpha_1,\cdots,-\alpha_{n_2}, 0,\cdots, 0,1)}{\|(-\alpha_1,\cdots, -\alpha_{n_2}, 0,\cdots, 0,1)\|}.$$
Here $\alpha_i\not\in\Q$, $\alpha_1, \cdots, \alpha_{n_2}$ are rationally independent
and $\|(-\alpha,1)\|:=\sqrt{\alpha_1^2 +\cdots + \alpha_{n_2}^2+1}$.
We firstly tackle this simplest case and 
we shall point out how to deal with the general case 
at the end of this subsection.

Firstly, Lemma \ref{lem:3.6211} becomes

\begin{lem}\label{lem:3.621}
If $u \in W_{loc}^{1,2}\left(\R^{n_2}\times \T^{n-n_2}\right)$ is minimal, then for any
$\phi\in W_{loc}^{1,2}\left(\R^{n_2}\times \T^{n-n_2}\right)$ with compact support,
\begin{equation}\label{eq:3.63}
\int_{\R^{n_2}\times \T^{n-n_2}}[L(u+\phi)-L(u)] \ud x \geq 0.
\end{equation}
\end{lem}

Moser showed that $\MMM_{\alpha}\neq \emptyset$ by a simple approximation method.
Moreover, Moser showed that there is a $u_{\alpha}\in\MMM(\bar{a}^{1})$.
Indeed, Moser showed that there is a $u_{\alpha}\in \MMM_{\alpha}$ such that
\begin{equation}\label{eq:moser6.521}
  u_{\alpha}(x+\bd{j})-\bar{\bd{j}}_{n+1}=u_{\alpha}(x)\quad\textrm{if}\quad \alpha\cdot \bd{j}-\bar{\bd{j}}_{n+1}=0.
\end{equation}
In other words, $\MT_{\bar{\bd{k}}}u_{\alpha}=u_{\alpha}$ for $\bar{\bd{k}}\in \bar{\Lambda}_2$.
One also need to prove that
\begin{equation}\label{eq:kdkdppp}
  \textrm{if $\bar{\bd{k}}\in\bar{\Lambda}_1=\Z^{n+1}$ and $\bar{\bd{k}}\cdot \bar{a}^1>0$, $\MT_{\bar{\bd{k}}}u_{\alpha}>u_{\alpha}$.}
\end{equation}
By Lemma \ref{prop:dk68956},
\eqref{eq:kdkdppp} holds since the rotation vector of $u_{\alpha}$ is $\alpha$.

Define $\MMM^{rec}(\bar{a}^1)$
as in \eqref{eq:133434134}.
By Lemma \ref{lem:dlffdf},
for any $u\in \MMM(\bar{a}^1)$,
$\MMM^{rec}(\bar{a}^1)\cup \{u\}$ is ordered.
Note that $u$ may be not contained in $\MMM^{rec}(u)$, 
for example in the case $u\in\MMM(\bar{a}^1)$ is an isolate point 
in its $\bar{\Lambda}_1$ orbit (cf. \cite[p. 111, line -11]{Bangert}).

If $\{u(\bd{0})\,|\, u\in\MMM^{rec}(\bar{a}^1)\}= \R$,
for any admissible $\bar{a}^1, \cdots, \bar{a}^t$ with $t>1$,
$\MMM(\bar{a}^1,\cdots, \bar{a}^t)= \emptyset$.
So in the following we
assume $\{u(\bd{0})\,|\, u\in\MMM^{rec}(\bar{a}^1)\}\neq \R$.
Then by Lemma \ref{lem:659712},
$\{u(\bd{0})\,|\, u\in\MMM^{rec}(\bar{a}^1)\}$ is a Cantor set. 
If $\MMM(\bar{a}^1)\setminus \MMM^{rec}(\bar{a}^1)\neq \emptyset$,
any $u\in\MMM(\bar{a}^1)\setminus \MMM^{rec}(\bar{a}^1) $ should lie in some adjacent pair of $\MMM^{rec}(\bar{a}^1)$.
Next we shall provide a variational construction for any adjacent pair, 
say $v<w$ of $\MMM^{rec}(\bar{a}^1)$.

To this end, suppose that $v < w\in \MMM^{rec}(\bar{a}^1)$ are adjacent,
i.e., there does not exist $u\in\MMM^{rec}(\bar{a}^1)\setminus\{v,w\}$ such that $v\leq u\leq w$.
Note if $v,w\in \MMM^{rec}(\bar{a}^1)$ are adjacent,
so is $\MT_{\bar{\bd{k}}}v, \MT_{\bar{\bd{k}}}w\in \MMM^{rec}(\bar{a}^1)$ for any $\bar{\bd{k}}\in\bar{\Lambda}_1=\Z^{n+1}$.
Then $v,w$ satisfy
\begin{equation}\label{eq:ddlpqqp}
  \textrm{if $\bar{\bd{k}}\in\bar{\Lambda}_1$ and $\bar{\bd{k}}\cdot \bar{a}^1>0$, then
$\MT_{\bar{\bd{k}}}v \geq w$.}
\end{equation}

Taking $\bar{\bd{k}}=\bar{\bd{e}}^{n+1}$ in \eqref{eq:ddlpqqp}
shows $v+1\geq w$.
So $0<w-v\leq 1$ and thus by Lemma \ref{lem:bangert11},
\begin{equation}\label{eq:dkkldldl}
 \norm{w-v}^{2}_{L^2(\R^{n_2}\times \T^{n-n_2})}\leq \norm{w-v}_{L^1(\R^{n_2}\times \T^{n-n_2})}\leq 1.
\end{equation}
Define
\begin{equation*}
  \Gamma_1:=\Gamma_1(v,w):=\{u\in W^{1,2}_{loc}(\R^{n_2}\times \T^{n-n_2} )\,|\, v\leq u\leq w\}.
\end{equation*}
For $u\in\Gamma_1$, by \eqref{eq:dkkldldl},
\begin{equation}\label{eq:dkkldlwwwwdl}
 \norm{u-v}^{2}_{L^2(\R^{n_2}\times \T^{n-n_2})}\leq \norm{u-v}_{L^1(\R^{n_2}\times \T^{n-n_2})}\leq 1.
\end{equation}
The following property is useful.
\begin{prop}\label{prop:kdk}
Suppose $u\in\Gamma_1\setminus\{v\}$ and $u$ is a classical solution of \eqref{eq:PDE}.
Then $u$ is WSI.
\end{prop}
\begin{proof}
Since $\MT_{\bar{\bd{k}}}u=u$ for $\bar{\bd{k}}\in \bar{\Lambda}_2 =\Z^{n+1}\cap \langle \bar{a}^1 \rangle ^{\perp}$,
it suffices to prove that
\begin{equation*}
  \left\{
    \begin{array}{ll}
      \MT_{\bar{\bd{k}}}u>u, & \textrm{for all $\bar{\bd{k}}\in\bar{\Lambda}_1$ with $\bar{\bd{k}}\cdot \bar{a}^1>0$,} \\
      \MT_{\bar{\bd{k}}}u<u, & \textrm{for all $\bar{\bd{k}}\in\bar{\Lambda}_1$ with $\bar{\bd{k}}\cdot \bar{a}^1<0$.}
    \end{array}
  \right.
\end{equation*}
Suppose $\bar{\bd{k}}\cdot \bar{a}^1> 0$.
Then $\MT_{\bar{\bd{k}}}u \geq \MT_{\bar{\bd{k}}}v \geq w\geq u$.
Since $u\neq v$,
there is $x\in \R^n$ such that $u(x)>v(x)$ 
and then $\MT_{\bar{\bd{k}}}u(x)>\MT_{\bar{\bd{k}}}v(x)$.
Thus $\MT_{\bar{\bd{k}}}u(x)>u(x)$.
By Lemma \ref{lem:rab}, $\MT_{\bar{\bd{k}}}u >u $.
The other case can be proved similarly.
\end{proof}
For $\bd{p}=(\bd{p}_1, \cdots, \bd{p}_{n_2})\in\Z^{n_2}$,
setting $T_{\bd{p}}:=T_{\bd{p}_1,\cdots, \bd{p}_{n_2}}:=[\bd{p}_1,\bd{p}_{1}+1]\times \cdots \times [\bd{p}_{n_2},\bd{p}_{n_2}+1]\times \T^{n-n_2}$, we have
\begin{equation}\label{eq:dkkldddlwwwwdl}
\lim_{|\bd{p}|\to \infty}\|u-v\|_{L^2(T_{\bd{p}})}=0=\lim_{|\bd{p}|\to \infty}\|u-w\|_{L^2(T_{\bd{p}})},
\end{equation}
where $|\bd{p}|:=|\bd{p}_1|+\cdots +|\bd{p}_{n_2}|$.

Next we introduce a functional on $\Gamma_1$.
For $\bd{k}\in \Z^{n_2}$,
define
\begin{equation*}
  J_{1,\bd{k}}(u):=\int_{T_{\bd{k}}}[L(u)-L(v)]\ud x.
\end{equation*}
For any $\bd{p}\leq \bd{q}$ (that is $\bd{p}_{\ell}\leq \bd{q}_{\ell}$ for all $1\leq\ell \leq n_2$),
set
\begin{equation*}
  T_{\bd{p},\bd{q}}:=\mathop{\cup}\limits_{\bd{k}_{n_2}=\bd{p}_{n_2}}^{\bd{q}_{n_2}}
  \cdots\mathop{\cup}\limits_{\bd{k}_2=\bd{p}_2}^{\bd{q}_2}
  \mathop{\cup}\limits_{\bd{k}_1=\bd{p}_1}^{\bd{q}_1}T_{\bd{k}}
\end{equation*}
and
\begin{equation*}
  J_{1;\bd{p},\bd{q}}(u):=\sum_{\bd{k}=\bd{p}}^{\bd{q}}J_{1,\bd{k}}(u)
  :=\left(\sum_{\bd{k}_{n_2}=\bd{p}_{n_2}}^{\bd{q}_{n_2}}\cdots\sum_{\bd{k}_2=\bd{p}_2}^{\bd{q}_2}
  \sum_{\bd{k}_1=\bd{p}_1}^{\bd{q}_1}\right)J_{1,\bd{k}}(u),
\end{equation*}
i.e.,
\begin{equation*}
  J_{1;\bd{p},\bd{q}}(u)=\int_{T_{\bd{p},\bd{q}}}[L(u)-L(v)]\ud x.
\end{equation*}

As in \cite{RS},
we shall prove $J_{1;\bd{p},\bd{q}}(u)$ is bounded from below with a lower bounded independent of $\bd{p},\bd{q}$.
The following proof follows from \cite[Proposition 2.8 and p. 38-39]{RS}.

\begin{prop}\label{prop:2.8}
Suppose $u\in\Gamma_1$ and $\bd{p}\leq \bd{q}\in  \Z^{n_2}$.
Then there is a constant $K_1=K_1(v,w,\alpha) \geq 0$ such that
\begin{equation*}
    J_{1;\bd{p},\bd{q}}(u)\geq -K_1.
\end{equation*}
\end{prop}
\begin{proof}
For $u\in\Gamma_1$,
we have
\begin{equation}\label{eq:ssddhj}
\begin{split}
&J_{1;\bd{p},\bd{q}}(u)\\
=&\int_{T_{\bd{p},\bd{q}}}[L(u)-L(v)]\ud x\\
=& \int_{T_{\bd{p},\bd{q}}}\left(\frac{1}{2}|\nabla (u-v)|^2+\nabla (u-v)\nabla v +[F(x,u)-F(x,v)]\right)\ud x\\
=& \frac{1}{2}\|\nabla (u-v)\|_{L^2(T_{\bd{p},\bd{q}})}^2+\int_{T_{\bd{p},\bd{q}}} \nabla (u-v)\nabla v \ud x
+\int_{T_{\bd{p},\bd{q}}}[F(x,u)-F(x,v)] \ud x\\
=:& \frac{1}{2}\|\nabla (u-v)\|_{L^2(T_{\bd{p},\bd{q}})}^2+I_1+I_2.
\end{split}
\end{equation}
Firstly we have
\begin{equation}\label{eq:4.5}
\begin{split}
  |I_2|=&\Big|\int_{T_{\bd{p},\bd{q}}}[F(x,u)-F(x,v)] \ud x\Big|\\
  \leq &\int_{T_{\bd{p},\bd{q}}}|F(x,u)-F(x,v)|\ud x\\
  \leq &\norm{F_u}_{L^{\infty}(\T^{n+1})}\int_{T_{\bd{p},\bd{q}}}(u-v)\ud x\\
  \leq &\norm{F_u}_{L^{\infty}(\T^{n+1})}\int_{T_{\bd{p},\bd{q}}}(w-v)\ud x,
  \end{split}
\end{equation}
thus by Lemma \ref{lem:bangert11},
\begin{equation*}
\int_{\R^{n_2}\times \T^{n-n_2}}|F(x,u)-F(x,v)|\ud x
\end{equation*}
exists and
$I_2\geq -\norm{F_u}_{L^{\infty}(\T^{n+1})}\cdot \norm{w-v}_{L^1 (\R^{n_2}\times \T^{n-n_2})}$.
Next we have
\begin{equation*}
\begin{aligned}
I_1=& \int_{T_{\bd{p},\bd{q}}} \nabla(u-v) \cdot \nabla v \ud x\\
=& \int_{\partial T_{\bd{p},\bd{q}}}(u-v) \frac{\partial v}{\partial \nu} \ud S -\int_{T_{\bd{p},\bd{q}}}(u-v) \Delta v \ud x\\
=&:I_{1,1}+I_{1,2}.
\end{aligned}
\end{equation*}
Since $\Delta v=F_u(x,v)$, the argument of \eqref{eq:4.5} shows that
\begin{equation*}
\int_{\R^{n_2}\times \T^{n-n_2}}(u-v) \Delta v \ud x
\end{equation*}
exists and
$$I_{1,2}\geq -\norm{F_u}_{L^{\infty}(\T^{n+1})}\cdot \norm{w-v}_{L^1 (\R^{n_2}\times \T^{n-n_2})}.$$
For $I_{1,1}$ we have
\begin{equation}\label{eq:ddd}
\begin{split}
  |I_{1,1}|\leq&\int_{\partial T_{\bd{p},\bd{q}}}| (u-v) \frac{\partial v}{\partial \nu} |\ud S\\
  \leq& \norm{\nabla v}_{L^{\infty}(\R^n)}\int_{\partial T_{\bd{p},\bd{q}}}(w-v) \ud S.
  \end{split}
\end{equation}

The derivative of $v$ is bounded by the following result of Moser.
\begin{lem}[{cf. \cite[Theorem 3.1]{moser}}]\label{lem:moser}
If $v\in\MMM_{\alpha}$,
then
$\|\nabla v\|_{L^{\infty}(\R^n)}\leq C(\alpha)$,
where $C(\alpha)$ is a constant only depending on the rotation vector $\alpha$.
\end{lem}

Since (cf. the proof of \cite[Lemma 4.5]{Bangert})
$$\int_{\partial T_{\bd{p},\bd{q}}}(w-v) \ud S
\leq \textrm{Volume of $\T^{n}$ in $\R^{n}$} = 1,
$$
by Lemma \ref{lem:moser} and \eqref{eq:ddd},
\begin{equation*}
  \lim_{\bd{p}\to -\infty \atop \bd{q}\to \infty}|I_{1,1}|
\end{equation*}
exists
and
$|I_{1,1}|\leq C(\alpha)$.
Here $\bd{p}\to -\infty$ (resp. $\bd{q}\to \infty$) means $\bd{p}_i \to -\infty$ (resp. $\bd{q}_i \to \infty$) for all $1\leq i\leq n_2$.
Thus
$J_{1;\bd{p},\bd{q}}(u)$ is bounded from below by
$$-K_1:= -C(\alpha)-2\norm{F_u}_{L^{\infty}(\T^{n+1})}\cdot \norm{w-v}_{L^1 (\R^{n_2}\times \T^{n-n_2})}.$$
\end{proof}

Now for $u\in\Gamma_1$, it is reasonable to define
$J_1(u):=\liminf_{\bd{p}\to -\infty \atop\bd{q}\to \infty}J_{1;\bd{p},\bd{q}}(u)$.
In fact, the proof of Proposition \ref{prop:2.8} shows 
$J_1(u)=\lim_{\bd{p}\to -\infty \atop\bd{q}\to \infty}J_{1;\bd{p},\bd{q}}(u)$.
It is easy to see that either $J_1(u)<\infty$ or $J_1(u)=\infty$.
In the latter case,
\begin{equation}\label{eq:4.8}
  J_{1}(u)=\infty \Longleftrightarrow \norm{\nabla (u-v)}_{L^2 (\R^{n_2}\times \T^{n-n_2})}=\infty.
\end{equation}
We have
\begin{lem}\label{lem:2.22}
If $u\in\Gamma_1$ and $\bd{p}\leq \bd{q}\in \Z^{n_2}$, then
\begin{equation}\label{eq:2.23}
  J_{1; \bd{p},\bd{q}}(u)\leq J_1(u)+M\cdot K_1,
\end{equation}
where
$M$ depends on the dimension of $V_2$.
\end{lem}
\begin{proof}
The proof of Lemma \ref{lem:2.22} is same to \cite[Lemma 2.22]{RS} and we omit it here.
\end{proof}
\begin{rem}\label{rem:kddfff}
The definition of $J_1$ depends on $v$ and $w$,
or $J_1$ defines on $\Gamma_1$.
Note $\MT_{\bar{\bd{k}}}v\not\in\Gamma_1$ for $\bar{\bd{k}}\in \bar{\Lambda}_1$ with $\bar{\bd{k}}\cdot \bar{a}^1\neq 0$.
Indeed, if $\bar{\bd{k}}\cdot \bar{a}_{1}>0$, $\MT_{\bar{\bd{k}}}v>w$.
Otherwise, $\bar{\bd{k}}\cdot \bar{a}_{1}<0$ implies $\MT_{\bar{\bd{k}}}v<v$.
So $J_1(\MT_{\bar{\bd{k}}}v)$ is not well-defined.
However, if $u\in \Gamma_1$ and $\tilde{u}$ is different with $u$ on a compact set,
$J_1(\tilde{u})$ is well-defined, even $\tilde{u}$ may be not contained in $\Gamma_1$.
Similar conclusion holds for $J_2$ in the next section.
\end{rem}

\begin{prop}\label{prop:2.241}
If $u\in\Gamma_1$ and $J_1(u)<\infty$, then
for $\bd{k}\in \Z^{n_2}$,
{\setlength\arraycolsep{2pt}
\begin{eqnarray}
  \lim_{|\bd{k} |\to \infty }J_{1, \bd{k}}(u) &= & 0,  \label{eq:2.251}\\
  \lim_{|\bd{k} |\to \infty}\left\|\MT_{\bd{k}}( u-v)\right\|_{W^{1,2}\left(T_{\bd{0}}\right)} &=& 0, \label{eq:2.261}
\end{eqnarray}}%
where $\MT_{\bd{k}}u(\cdot):=u(\cdot-(\bd{k},\bd{0}))$ with $(\bd{k},\bd{0})\in \Z^{n_2}\times \Z^{n-n_2}$.
Furthermore, for $p\in\N$, set $\bd{p}=(p,\cdots, p)\in\N^{n_2}$.
Then
{\setlength\arraycolsep{2pt}
\begin{eqnarray}
\lim_{p\to\infty}[J_1(u)-J_{1;-\bd{p}, \bd{p}}(u)]&=&0,\\
\lim_{p\to\infty}\sum_{\bd{k}\in\Z^{n_2}\setminus A_p}\|\MT_{\bd{k}}( u-v)\|_{W^{1,2}\left(T_{\bd{0}}\right)}&=&0, \label{eq:621644}
\end{eqnarray}}%
where $A_p=\{\bd{j}\in\Z^{n_2}\,|\, |\bd{j}_{\ell}|\leq p, \,\, \textrm{for } \ell=1, \cdots, n_2\}$.
\end{prop}
\begin{proof}
By \eqref{eq:dkkldlwwwwdl},
  $\norm{u-v}_{L^2(\R^{n_2}\times \T^{n-n_2})}<\infty.$
Then by \eqref{eq:4.8},
Proposition \ref{prop:2.241} is proved.
\end{proof}


In the following, we can prove $J_1(w)=0$, so by Proposition \ref{prop:2.241},
$$\lim_{p\to\infty}\sum_{\bd{k}\in\Z^{n_2}\setminus A_p}\|\MT_{\bd{k}}( w-v)\|_{W^{1,2}\left(T_{\bd{0}}\right)}=0.$$
Following the lines of Rabinowitz--Stredulinsky \cite{RS}, we define
$$c_1:=c_1(v,w):=\inf_{u\in\Gamma_1}J_1(u).$$
By Proposition \ref{prop:2.8},
$c_1$ is well-defined and as one may expect, we have
\begin{prop}\label{prop:dd}
$c_1=0$.
\end{prop}
\begin{proof}
We shall show
for all $u\in \Gamma_1$,
\begin{equation}\label{eq:kk}
  J_1(v)\leq J_1(u).
\end{equation}
Without loss of generality,
suppose $J_1(u)<\infty$.
Set $x=(\tilde{x},\hat{x})\in \R^{n_2}\times \T^{n-n_2}$.
For $p\in\N$,
define $|\tilde{x}|:=\max_{1\leq i \leq n_2} |\tilde{x}_i|$ 
and
\begin{equation*}
  u_p:=\left\{
        \begin{array}{ll}
          u, & |\tilde{x}|\leq p,\\
          u+(|\tilde{x}|-p)(v-u), & p< |\tilde{x}|< p+1 , \\
          v, & |\tilde{x}|\geq p+1.
        \end{array}
      \right.
\end{equation*}
Since $v$ is minimal, by Lemma \ref{lem:3.621},
  $J_1(v)\leq J_1(u_p).$
By Proposition \ref{prop:2.241}, $J_1(u_p)\to J_1(u)$ as $p\to\infty.$
Thus \eqref{eq:kk} holds and $c_1=J_1(v)=0$.
\end{proof}
\begin{rem}\label{rem:kdk}
Note that in the proof of Proposition \ref{prop:dd},
if we replace $v$ by any $\tilde{v}\in \MMM(\bar{a}^1)\cap \Gamma_1$, the proof also holds.
In particular,
replace $v,u$ in \eqref{eq:kk} by $\tilde{v},v$, respectively,
then we have
$J_1(\tilde{v})=0$.
Thus in the definition of $J_1$, we can replace $v$ by any $\tilde{v}\in\MMM(\bar{a}^1)\cap \Gamma_1$.
\end{rem}
By the proof of Proposition \ref{prop:2.8}, we have
\begin{equation*}
\begin{split}
  J_1(u)=\frac{1}{2}&\norm{\nabla (u-v)}^2 _{L^2 (\R^{n_2}\times \T^{n-n_2})}- \int_{\R^{n_2}\times \T^{n-n_2}} (u-v) \Delta v \ud x\\
  &+ \int_{\R^{n_2}\times \T^{n-n_2}}[F(x,u)-F(x,v)]\ud x.
  \end{split}
\end{equation*}
Define
\begin{equation}\label{eq:mm1}
  \MM_1:=\MM_1(v,w):=\{u\in\Gamma_1(v,w)\,|\, J_1(u)=0\}.
\end{equation}
The above argument shows $\MMM(\bar{a}^1)\cap \Gamma_1\subset \MM_1$.
We shall prove:

\begin{prop}\label{prop:ped}
$\MM_1=\MMM(\bar{a}^1)\cap \Gamma_1$.
\end{prop}
This property is interesting since it gives a representation of the elements of $\MMM(\bar{a}^1)\setminus \MMM^{rec}(\bar{a}^1)$.
Recall that all the elements of $\MMM^{rec}(\bar{a}^1)$ can be approximated by minimal and WSI periodic solutions (cf. \cite[Corollary 5.4]{bangert-uniq}).

\bigskip
\noindent
\emph{Proof.}
What is needed to prove is
that any $u\in\MM_1$ belongs to $\MMM(\bar{a}^1)$.
First assume that $u\in\MM_1$ is minimal, then
$u$ is a solution of \eqref{eq:PDE} and $u\in C^2(\R^n)$.
If $u=v$, the conclusion is obvious.
Next we assume $u\neq v$.
By Proposition \ref{prop:kdk}, $u$ is WSI.
Hence by Lemma \ref{lem:dldloplrrtyto}, $\bar{a}^1(u)=\bar{a}^1$.
Note $\MT_{\bar{\bd{k}}}u=u$ for all $\bar{\bd{k}}\in\bar{\Lambda}_2$.
Thus by Lemma \ref{prop:dk68956},
$u\in \MMM(\bar{a}^1)$.
To complete the proof of Proposition \ref{prop:ped}, we need to show that $u$ is minimal.
To see this, we need some preliminaries.

\begin{prop}\label{prop:2.50}
Let $\MY\subset \Gamma_1$ and
suppose $\MY$ satisfies the following condition:
\begin{enumerate}
  \item[$(Y^1 _1)$] \label{eq:Y11}
Let $u_k,u\in\MY$ and $p\in\N$.
Suppose $u_k-v\to U-v$ weakly in $W^{1,2}_{loc}(\R^{n_2}\times \T^{n-n_2})$ as $k\to\infty$.
Set $x=(\tilde{x},\hat{x})\in\R^{n_2}\times \T^{n-n_2}$
and $|\tilde{x}|:=\max_{1\leq i \leq n_2} |\tilde{x}_i|$.
Define $\chi_{p}:= \chi_{p}(u, U)$ by
\begin{equation}\label{eq:dddd12dd}
\chi_{p}=\left\{\begin{array}{ll}
U, &  |\tilde{x}| \leq p, \\
U+(u-U)(|\tilde{x}|-p), & p < |\tilde{x}| < p+1, \\
u, & |\tilde{x}| \geq p+1 .\\
\end{array}\right.
\end{equation}
Then $\chi_{p}(u,U)\in\MY$ for all large $p$ (independent of $u$).
\end{enumerate}
Define
\begin{equation}\label{eq:2.51}
c(\MY)=\inf _{u \in \mathcal{Y}} J_{1}(u).
\end{equation}
If $c(\MY)\in\R$ and $(u_k)$ is a minimizing sequence for \eqref{eq:2.51}, then there is a $U\in\Gamma_1$ such that
along a subsequence $u_k-v\to U-v$ in $W^{1,2}_{loc}(\R^{n_2}\times \T^{n-n_2})$ as $k\to\infty$.
\end{prop}
\begin{proof}
The proof is similar to that of \cite[Proposition 2.50]{RS} 
so we omit it here.
\end{proof}

\begin{prop}\label{prop:2.64}
Assume $\MY \subset \Gamma_1$ satisfies ($Y_1 ^1$).
Let $c(\MY)$ and $(u_k)$ be defined as in Proposition \ref{prop:2.50}.
If for some $r\in (0,\frac{1}{2})$, some $z\in\R^n$, all smooth $\phi$ with support in $B_{r}(z):=\left\{x \in \mathbb{R}^{n}|| x-z |<r\right\}$ and associated $t_{0}(\phi)>0$,
\begin{equation}\label{eq:2.65}
c(\MY) \leq J_{1}\left(u_{k}+t \phi\right)+\delta_{k}
\end{equation}
for all $|t| \leq t_{0}(\phi)$, where $\delta_{k}=\delta_{k}(\phi) \rightarrow 0$ as $k\to \infty$.
Then the weak limit $U$ of $u_k$ (in the sense of $u_k-v\to U-v$ weakly in $W^{1,2}_{loc}(\R^{n_2}\times \T^{n-n_2})$ as $k\to\infty$) satisfies \eqref{eq:PDE} in $B_{r}(z)$.
\end{prop}
\begin{proof}
The proof is similar to that of \cite[Proposition 2.64]{RS} 
so we omit it here.
\end{proof}

Let $\bd{l}:=(\bd{l}_{{n_2}+1},\cdots, \bd{l}_n)\in\N^{n-{n_2}}$
and
\begin{equation*}
\begin{split}
&W^{1,2}_{loc}(\R^{n_2}\times \T_{\bd{l}}^{n-{n_2}})\\
:=&\{u\in W^{1,2}_{loc}(\R^n)\,|\, u(x_1,\cdots, x_{n_2},x_{{n_2}+1}+\bd{l}_{{n_2}+1},\cdots, x_n+\bd{l}_{n})=u(x)\}.
\end{split}
\end{equation*}
Set
\begin{equation}\label{eq:gamma}
  \Gamma_1 ^{\bd{l}}:=\Gamma_1 ^{\bd{l}}(v,w):=\{u\in W^{1,2}_{loc}(\R^{n_2}\times \T_{\bd{l}}^{n-{n_2}})\,|\, v\leq u\leq w \}.
\end{equation}
For
$u\in \Gamma_1 ^{\bd{l}}$,
set
\begin{equation}\label{eq:j1l}
  J_1 ^{\bd{l}}(u):=\int_{\R^{n_2}\times \T_{\bd{l}}^{n-{n_2}}}[L(u)-L(v)]\ud x
\end{equation}
and
\begin{equation*}
  c_1 ^{\bd{l}}:=c_1 ^{\bd{l}}(v,w):=\inf_{u\in\Gamma_1 ^{\bd{l}}} J_1 ^{\bd{l}}(u).
\end{equation*}
By the proof of Proposition \ref{prop:2.8},
$c_1 ^{\bd{l}}$ is well-defined and $-(\prod_{i=n_2+1}^{n}\bd{l}_i)K_1\leq c_1 ^{\bd{l}}\leq 0$.

\begin{rem}\label{rem:lwwwwfkfk}
If we replace $\T^{n-n_2}$ by $\T^{n-n_2}_{\bd{l}}$,
Lemma \ref{lem:3.621}
still holds.
\end{rem}

\begin{rem}\label{rem:lfkfk}
If we replace
$\Gamma_1$, $J_1$ by $\Gamma_1 ^{\bd{l}}$, $J_1 ^{\bd{l}}$, respectively,
and assume $\MY\subset \Gamma_1 ^{\bd{l}}$ and
$c^{\bd{l}}(\MY):=\inf_{u\in\MY} J_1 ^{\bd{l}}(u)$,
then
Propositions \ref{prop:2.50} and \ref{prop:2.64} still hold.
\end{rem}
\begin{rem}\label{rem:kdkjv}
Proceeding as in the proof of
Proposition \ref{prop:dd} with obvious modifications,
$c_{1}^{\bd{l}}=0$.
\end{rem}

\begin{prop}\label{prop:2.2}
Let
\begin{equation*}
  \MM_1 ^{\bd{l}}:=\{u\in \Gamma_1 ^{\bd{l}}\,|\, J_1 ^{\bd{l}}(u)=c_1 ^{\bd{l}}=0\}.
\end{equation*}
Then $\MM_1 ^{\bd{l}}\neq \emptyset$ and $\MM_1 ^{\bd{l}}=\MM_1 $.
\end{prop}

\begin{proof}
Since $v,w\in \MM_1 ^{\bd{l}}$, $\MM_1 ^{\bd{l}}\neq \emptyset$.
To show $\MM_1 ^{\bd{l}}=\MM_1 $,
following \cite[Proposition 2.2]{RS}, we need to prove
\begin{enumerate}
  \item \label{prop:2.2-1} any $u\in\MM_1 ^{\bd{l}}$ is a classical solution of \eqref{eq:PDE};
  \item \label{prop:2.2-2} $\MM_1 ^{\bd{l}}$ is an ordered set;
  \item \label{prop:2.2-3} any $u\in\MM_1 ^{\bd{l}}$ is contained in $\MM_1$.
\end{enumerate}
Firstly \eqref{prop:2.2-1} is ensured by Proposition \ref{prop:2.64}.
Indeed, assume that there is a minimizing sequence $u_k\in\Gamma_1 ^{\bd{l}}$
such that $u_k-v\to u-v$ weakly in $W^{1,2}_{loc}(\R^{n_2}\times \T_{\bd{l}}^{n-n_2})$ as $k\to\infty$.
Let $r=\frac{1}{4}$. For any $z\in\R^n$, define
$\phi$ and $B_r(z)$ as in Proposition \ref{prop:2.64}.
We shall prove \eqref{eq:2.65} holds for $t_0(\phi)=1$ (with $J_1$ replaced by $J_1 ^{\bd{l}}$).
Notice that for $|t|\leq 1$,
\begin{equation*}
\begin{split}
  &J_1 ^{\bd{l}}\left(\max\left(\min(u_k+t\phi, w),v\right)\right)+
  J_1 ^{\bd{l}}\left(\min\left(\min(u_k+t\phi, w),v\right)\right)\\
  =&J_1 ^{\bd{l}}(\min(u_k+t\phi, w))+J_1 ^{\bd{l}}(v)\\
  =&J_1 ^{\bd{l}}(\min(u_k+t\phi, w)),
  \end{split}
\end{equation*}
and
\begin{equation*}
\begin{split}
&J_1 ^{\bd{l}}(\max(u_k+t\phi, w))+J_1 ^{\bd{l}}(\min(u_k+t\phi, w))\\
=&J_1 ^{\bd{l}}(u_k+t\phi)+J_1 ^{\bd{l}}(w)\\
=&J_1 ^{\bd{l}}(u_k+t\phi).
 \end{split}
 \end{equation*}
By Lemma \ref{lem:3.621} and Remark \ref{rem:lwwwwfkfk} and $r=\frac{1}{4}<1$,
\begin{equation*}
  J_1 ^{\bd{l}}\left(\min\left(\min(u_k+t\phi, w),v\right)\right) \geq 0,
   \quad J_1 ^{\bd{l}}(\max(u_k+t\phi, w)) \geq 0.
\end{equation*}
Thus we have
\begin{equation}\label{eq:1234dd}
J_1 ^{\bd{l}}(u_k+t\phi)\geq J_1 ^{\bd{l}}\left( \max\left(\min(u_k+t\phi, w),v\right) \right).
\end{equation}
Since $\max\left(\min(u_k+t\phi, w),v\right)\in \Gamma_1 ^{\bd{l}}$,
$J_1 ^{\bd{l}}(u_k+t\phi)\geq c^{\bd{l}}_{1}$.
So by Proposition \ref{prop:2.64} and Remark \ref{rem:lfkfk},
$u$ is a classical solution of \eqref{eq:PDE}.
For any $\tilde{u}\in\MM_1 ^{\bd{l}}$,
taking the minimizing sequence $u_k$ by $u_k\equiv\tilde{u}$,
the above argument shows that $\tilde{u}$ is a classical solution of \eqref{eq:PDE}.

The proofs of \eqref{prop:2.2-2} and \eqref{prop:2.2-3} are similar to 
that of \cite[Proposition 2.2]{RS}, so we omit it here.
\end{proof}

\noindent
\emph{Continue the proof of Proposition \ref{prop:ped}.}\label{page}
It suffices to show that $u$ is minimal.
Proceeding as in the proof of \cite[Theorem 3.60]{RS} shows $u$ is minimal.
Indeed,
suppose by contradiction
$u$ is not minimal.
By the definition of minimal,
there is a bounded domain $\Omega\subset \R^n$ with smooth boundary
and a smooth function $\phi$ with support contained in $\Omega$
such that
\begin{equation}\label{eq:dkkgf}
\int_{\Omega}[L(u)-L(u+\phi)]\ud x >0
\end{equation}
Without loss of generality, assume
$\Omega\subset [\bd{p},\bd{q}]\times [0, \bd{l}_{n_2+1}]\times\cdots\times [0,\bd{l}_n]$ for some $\bd{p}<\bd{q}\in\Z^{n_2}$ and $\bd{l}=(\bd{l}_{n_2+1},\cdots,\bd{l}_n)\in\N^{n-n_2}$.
By the periodicity of $u$ in $x_{n_2+1},\cdots,x_n$,
$$J_{1}^{\bd{l}}(u)=(\prod_{i={n_2+1}}^{n}\bd{l}_i)J_1(u)=0=c_{1}^{\bd{l}}.$$
By Remark \ref{rem:kddfff},
$J_{1}^{\bd{l}}(u+\phi)$ is well-defined and
$$J_{1}^{\bd{l}}(u+\phi)=\int_{\R^{n_2}\times \T_{\bd{l}}^{n-n_2}}[L(u+\phi)-L(v)]\ud x.$$
So \eqref{eq:dkkgf} implies $J_{1}^{\bd{l}}(u+\phi)< J_{1}^{\bd{l}}(u)=c_{1}^{\bd{l}}$.
Thus the argument of \eqref{eq:1234dd} shows
$$c_{1}^{\bd{l}}>J_{1}^{\bd{l}}(u+\phi)\geq J_1 ^{\bd{l}}\left( \max\left(\min(u+\phi, w),v\right) \right),
$$
contrary to $\max\left(\min(u+\phi, w),v\right)\in \Gamma_1 ^{\bd{l}}$ and the definition of $c_{1}^{\bd{l}}$.
\qed

\begin{rem}
By Propositions \ref{prop:ped} and \ref{prop:2.2}, $\MM_1 ^{\bd{l}}=\MM_1 =\MMM(\bar{a}^1)\cap \Gamma_1$
and $\MM_1$ is ordered.
\end{rem}

\begin{thm}\label{thm:ldfjldf}
Denote by $\langle v,w\rangle$ an adjacent pair of $\MMM^{rec}(\bar{a}^1)$ satisfying $v<w$ and
let $\mathcal{A}_1$ to be the set which consists of such pairs.
Then
$$\MMM(\bar{a}^1)=\MMM^{rec}(\bar{a}^1) \cup \left[ \mathop{\cup}\limits_{\langle v,w\rangle\in\mathcal{A}_1}\MM_1 (v,w)\right]$$
and so $\MMM(\bar{a}^1)$ is ordered.
\end{thm}

To conclude this subsection, let us point out how to deal with the general case.
Firstly set $\bd{s}=(\bd{s}_{n_2+1}, \cdots, \bd{s}_{n})\in \Z^{n-n_2}$ 
and $\bd{r}=(\bd{r}_{n_2+1}, \cdots, \bd{r}_{n})\in \N^{n-n_2}$.
If $$\bar{a}^1=\frac{(-\alpha_1,\cdots, -\alpha_{n},1)}{\|(-\alpha_1,\cdots, -\alpha_{n},1)\|}$$
with $\alpha_i\not\in\Q$ for $1\leq i\leq n_2$ and $\alpha_1, \cdots, \alpha_{n_2}$ are rationally independent,
and $\alpha_j=\frac{\bd{s}_j}{\bd{r}_j}\in\Q$ for $n_2+1\leq j\leq n$ and $\bd{s}_j, \bd{r}_j$ are relatively prime.
Then following \cite[Section 5.3]{RS},
define
$\Gamma_1:=\Gamma_{1}^{\bd{r}}+\sum_{k=n_2+1}^{n}\alpha_k \cdot x_k$ and
\begin{equation}\label{eq:5.26}
  c_{1}^{\bd{r}, \bd{s}}:=\inf _{u \in \Gamma_{1}^{\bd{r}}} J_{1}^{\bd{r}}(u+ \sum_{k=n_2+1}^{n}\alpha_k \cdot x_k),
\end{equation}
where $\Gamma_{1}^{\bd{r}}$ and $J_{1}^{\bd{r}}$ are defined as in \eqref{eq:gamma} and \eqref{eq:j1l}, respectively.
Set
\begin{equation*}
\mathcal{M}_{1}^{\bd{r}, \bd{s}}=
\left\{u+\sum_{k=n_2+1}^{n}\alpha_k \cdot x_k 
\Big| 
u \in \Gamma_{1}^{\bd{r}} \text { and } J_{1}^{\bd{r}}\left(u+\sum_{k=n_2+1}^{n}\alpha_k \cdot x_k\right)=c_{1}^{\bd{r}, \bd{s}}\right\}.
\end{equation*}
Proceeding as in Section \ref{sec:3.2}, 
we obtain the variational construction of elements of $\MMM(\bar{a}^1)\cap \Gamma_1$.
The reader is referred to \cite[Section 5.3]{RS} for more details related to this case.

Next for the general case, we shall reduce it to the above case by changing coordinate systems.
Since $rank(\bar{\Lambda}_2)=n-n_2\in [1,n-1]$,
$\dim(V_2)=rank(\Lambda_2)=rank(\bar{\Lambda}_2)=n-n_2$.
Recall that $\bar{\Lambda}_2=\Z^{n+1}\cap \langle \bar{a}^1\rangle ^{\perp}$ and
$V_2=span (\Lambda_2)=\pi(span (\bar{\Lambda}_2))$.
Consider the orthogonal decomposition of $\R^n=V_2 ^{\perp}\times V_2 $.

In the subspace $V_2$, choose $n-n_2$ orthogonal vectors $\omega^j\in\Z^n$, $n_2 +1\leq j\leq n$.
Thus $\omega^j=\sum_{k=1}^{n}\alpha_{kj}\bd{e}^k$ with $\alpha_{kj}\in\Z$, $n_2+1\leq j\leq n$, $1\leq k\leq n$.
For fixed $j$, 
it can be assumed that the components $\alpha_{kj}$ of $\omega^j$
have no common factor.
Now suppose $(\alpha_1,\cdots,\alpha_n)-\sum_{j=n_2+1}^{n}c_j\omega^j=:(\beta_1,\cdots,\beta_{n})\in V_2 ^{\perp}\setminus\{\bd{0}\}$ 
with $c_j\in \Q$, $n_2+1\leq j\leq n$.

In the subspace $V_2 ^{\perp}$,
fix any orthogonal vectors $\omega^j\in\Z^n$, $1 \leq j\leq n_2$.
Then there are $c_j\in \R\setminus \Q$, $1 \leq j\leq n_2$, such that
$(\beta_1,\cdots,\beta_{n})=\sum_{j=1}^{n_2}c_j \omega^j$.
We claim that
\begin{equation}\label{eq:claim:kddk}
\textrm{$c_j$, $1 \leq j\leq n_2$,
are rationally independent.}
\end{equation}
Indeed, suppose by contradiction that there is some $c_j$ (without loss of generality, assume $j=1$ and $n_2\geq 2$) such that
\begin{equation}\label{eq:dddddd}
c_1=\sum_{i=2}^{n_2}d_{1i}c_i
\end{equation}
with $d_{1i}\in\Q$ and not all $d_{1i}=0$.
First suppose
\begin{equation}\label{eq:claimllll}
  \textrm{$d_{1i}\neq 0$ for all $i=2,\cdots,n_2$.}
\end{equation}
Thus
\begin{equation*}
\begin{split}
  (\beta_1,\cdots,\beta_{n}) &=\sum_{j=1}^{n_2}c_j \omega^j \\
    &=c_1\omega^1+\sum_{i=2}^{n_2}c_i d_{1i} \frac{1}{d_{1i}}\omega^i\\
    &=c_1\omega^1+\left(c_2 d_{12}
    +\sum_{i=3}^{n_2}c_i d_{1i}\right) \frac{1}{d_{12}}\omega^2
    -\sum_{i=3}^{n_2}c_i d_{1i}\frac{1}{d_{12}}\omega^2
    +\sum_{i=3}^{n_2}c_i d_{1i} \frac{1}{d_{1i}}\omega^i\\
    &=c_1\left(\omega^1+\frac{1}{d_{12}}\omega^2\right)
    +\sum_{i=3}^{n_2}c_i d_{1i}\left[ \frac{1}{d_{1i}}\omega^i-\frac{1}{d_{12}}\omega^2\right].
\end{split}
\end{equation*}
But
\begin{equation}\label{eq:basis}
\omega^1+\frac{1}{d_{12}}\omega^2, \frac{1}{d_{13}}\omega^3-\frac{1}{d_{12}}\omega^2, \cdots, \frac{1}{d_{1n_2}}\omega^{n_2}-\frac{1}{d_{12}}\omega^2
\end{equation}
are linearly independent,
contrary to $\dim(V_2 ^{\perp})=n_2$.

If \eqref{eq:claimllll} is not hold, then there is/are $i$ such that $d_{1i}=0$.
Without loss of generality,
assume $c_1=\sum_{i=2}^{\ell}d_{1i}c_i$ with $2\leq \ell \leq n_2-1$.
Then
\begin{equation*}
\begin{split}
  (\beta_1,\cdots,\beta_{n+1}) &=\sum_{j=1}^{\ell}c_j \omega^j +\sum_{j=\ell+1}^{n_2}c_j \omega^j\\
    &=c_1\left(\omega^1+\frac{1}{d_{12}}\omega^2\right)
    +\sum_{i=3}^{\ell}c_i d_{1i}\left( \frac{1}{d_{1i}}\omega^i-\frac{1}{d_{12}}\omega^2\right)+\sum_{j=\ell+1}^{n_2}c_j \omega^j.
\end{split}
\end{equation*}
But
\begin{equation*}
\omega^1+\frac{1}{d_{12}}\omega^2, \frac{1}{d_{13}}\omega^3-\frac{1}{d_{12}}\omega^2, \cdots, \frac{1}{d_{1\ell}}\omega^{\ell}-\frac{1}{d_{12}}\omega^2, \omega^{\ell+1}, \cdots, \omega^{n_2}
\end{equation*}
are linearly independent,
again contrary to $\dim(V_2 ^{\perp})=n_2$.
Hence \eqref{eq:claim:kddk} is proved.

Now as in the subspace $V_2$, we set
$\omega^j=\sum_{k=1}^{n}\alpha_{kj}\bd{e}^k$ with $\alpha_{kj}\in\Z$, $1\leq j\leq n_2$, $1\leq k\leq n$,
and $\alpha_{kj}\in\Z$ have no common factor.
So $\{\omega^j \,|\, 1\leq j\leq n\}$ is an orthogonal basis of $\R^n$.
Let $B:=(\alpha_{ij})_{i,j=1}^{n}$
and $\lambda_k :=\sum_{j=1}^{n}\alpha_{jk}^2$ ($=\norm{\omega^k}^2$) for $1\leq k\leq n$.
Then
\begin{equation*}
  (\omega^1,\cdots, \omega^n)=(\bd{e}^1,\cdots,\bd{e}^n)B.
\end{equation*}
Assume $\xi=x_1 \bd{e}^1+\cdots x_n \bd{e}^n=y_1\omega^1+\cdots +y_n \omega^n$,
then
\begin{equation}\label{eq:relation}
\left\{
  \begin{array}{l}
    (x_1,\cdots, x_n)=(y_1,\cdots, y_n)B^{T}, \\
    (y_1,\cdots, y_n)=(x_1,\cdots, x_n)(B^{T})^{-1}.
  \end{array}
\right.
\end{equation}
Here $B^{T}$ is the transpose of $B$
and $(B^{T})^{-1}$ is the inverse matrix of $B^{T}$.
By the definition of $B$ and the fact that $\{\omega^i\}_{i=1}^{n}$ is an orthogonal basis,
$B^{T}B=diag(\lambda_1 , \cdots, \lambda_n )$,
where `diag' means a diagonal matrix.
So
\begin{equation}\label{eq:aboutB}
  (B^{T})^{-1}=B\cdot diag(\lambda_1 ^{-1} , \cdots, \lambda_n ^{-1}).
\end{equation}

If $u^*=u^*(x)$ satisfies
\eqref{eq:PDE}, then $u=u(y):=u^* (y B^{T})$ satisfies
\begin{equation}\label{eq:dfdfdv001}
  -\Big(\frac{1}{\lambda_1}  \frac{\partial  ^2 u}{\partial y_1 ^2}
  + \cdots +\frac{1}{\lambda_n}  \frac{\partial  ^2 u}{\partial y_n ^2}\Big)
  +\bar{F}_u(y, u)=0,
\end{equation}
with $\bar{F}_u(y,u):=F_u(y B^{T}, u)$.
Since $F$ satisfies \eqref{eq:F1-F2} and $\alpha_{ij}\in\Z$,
$\bar{F}$ also satisfies \eqref{eq:F1-F2}.
Although there are some additional constants $\lambda_k$ in \eqref{eq:dfdfdv001},
they do not affect our analysis.
By the construction of $B$,
if $u^*=u^*(x)$ satisfies \eqref{eq:PDE} with rotation vector $\alpha$,
then
$$
M\geq |u^*(x)-\alpha \cdot x|= |u^*(yB^{T})-\alpha \cdot yB^{T}|=|u(y)-\alpha \cdot y B^{T}|=|u(y)-\alpha B \cdot y |,
$$
which implies that the  rotation vector of $u$ is $\alpha B$ 
in the coordinate system $\omega^1,\cdots, \omega^n$.
For the rotation vector $\alpha = (\alpha_1,\cdots, \alpha_n)$,
we have
$\alpha_1 \bd{e}^1+\cdots \alpha_n \bd{e}^n= c_1 \omega^1 +\cdots +c_n \omega^n$.
Thus by \eqref{eq:relation} and \eqref{eq:aboutB},
$$(c_1,\cdots,c_n)=(\alpha_1,\cdots,\alpha_n)(B^T)^{-1}=(\alpha_1,\cdots,\alpha_n)B\cdot diag(\lambda_1 ^{-1} , \cdots, \lambda_n ^{-1}).$$
Hence the rotation vector of $u$ is $\alpha B=(c_1,\cdots,c_n)\cdot diag(\lambda_1  , \cdots, \lambda_n ).$
Notice that on the one hand $\lambda_k\in\N$ for all $1\leq k\leq n$,
and on the other hand
$c_i\not\in\Q$ for $1\leq i\leq n_2$ and $c_1, \cdots, c_{n_2}$ are rationally independent,
and $c_j\in\Q$ for $n_2+1\leq j\leq n$.
So we have reduced the general case to the one we already
know how to deal with.
After analyzing in the new coordinate system $\omega^1,\cdots, \omega^n$,
using $u^*(x)=u(y)$
we can return back to the coordinate system $\bd{e}^1,\cdots,\bd{e}^n$.


\subsection{The case $rank(\bar{\Lambda}_2)=0$}\label{sec:3.3}
\

In this case,
$\bar{a}^1$ is rationally independent.
Although in this case there is no basic heteroclinic solutions in higher dimensions,
it is interesting to give the variational construction of adjacent pairs in $\MMM^{rec}(\bar{a}^1)$.
Since given gap conditions,
more complicated solutions can be established by variational methods,
for instance, mountain pass solutions.

Note that in this case,
$n=n_2$.
Replace $\R^{n_2}\times \T^{n-n_2}$ by $\R^n$.
Lemma \ref{lem:3.621} is equivalent to the definition of minimal.
Define $\MMM^{rec}(\bar{a}^1)$ as in \eqref{eq:133434134}.
Lemma \ref{lem:bangert11}
should be replaced by the following lemma.
\begin{lem}[{cf. \cite[Lemma 4.5]{bangert-uniq}}]\label{lem:moser2}
Suppose $\bar{a}^1$ is rationally independent and $v<w$ are adjacent in $\MMM^{rec}(\bar{a}^1)$.
Then
\begin{equation*}
  \int_{\R^n}(w-v)\ud x\leq 1
\end{equation*}
and
\begin{equation*}
  \int_{\R^n}(w-v)^2\ud x\leq 1.
\end{equation*}
\end{lem}

Now
proceeding as in Section \ref{sec:3.2},
we can obtain the same conclusions.
For example, the proof of Proposition \ref{prop:kdk} is easier,
since $\bar{\Lambda}_2=\emptyset$.
To prove a property similar to
Proposition \ref{prop:ped},
one should show that $u\in\MM_1$ is minimal.
This is easier than Proposition \ref{prop:ped}.
Indeed,
in the case of $\bar{a}^1$ is rationally independent,
suppose $u$ is not minimal, then there exists a $\phi$ with compact support such that
\begin{equation*}
  \int_{\R^n}[L(u)-L(u+\phi)]\ud x >0.
\end{equation*}
Thus
\begin{equation}\label{eq:kldfsjglkafdg}
\begin{split}
  0=&J_1(u)=\int_{\R^n}[L(u)-L(v)]\ud x\\
  >&\int_{\R^n} [L(u+\phi)-L(v)]\ud x=J_1(u+\phi).
  \end{split}
\end{equation}
If $u+\phi\in\Gamma_1$,
\eqref{eq:kldfsjglkafdg}
contradicts Proposition \ref{prop:dd}.
If $u+\phi\not\in\Gamma_1$,
the arguments of \eqref{eq:1234dd} show $J_1(u+\phi)\geq J_1(\max(\min(u+\phi,w),v))\geq 0$,
again a contradiction.
So $u$ is minimal.


\section{The second renormalized  functional and basic heteroclinic solution}\label{sec:4}


\subsection{The second renormalized  functional for the case $rank(\bar{\Lambda}_3)=rank(\bar{\Lambda}_2)-1$}
\

If $rank(\bar{\Lambda}_2)= rank(\bar{\Lambda}_1)-1$,
then the case $rank(\bar{\Lambda}_3)=rank(\bar{\Lambda}_2)-1$ has been treated 
by Rabinowitz and Stredulinsky (see \cite[Chapter 4]{RS}).

If $rank(\bar{\Lambda}_2)=0$, then $\bar{\Lambda}_3=\emptyset$.
There do not exist basic heteroclinic solutions.

Next we assume $1\leq rank(\bar{\Lambda}_2)\leq rank(\bar{\Lambda}_1)-2$.
As in Section \ref{sec:3}, without loss of generality,
we assume
$$\bar{a}^1=\frac{(-\alpha_1,\cdots, -\alpha_{n_2}, 0,\cdots,0,1)}{\|(-\alpha_1,\cdots, -\alpha_{n_2}, 0,\cdots,0,1)\|}$$
and $\bar{a}^2=\bar{\bd{e}}^{n_2+1}$, i.e., the $(n_2+1)$-th component is $1$ and other components $0$.
Suppose the following gap condition:
$$ v<w\in\MMM^{rec}(\bar{a}^1) \textrm{ are an adjacent pair}.
$$

Set $S_i:=\R^{n_2}\times [i,i+1]\times \T^{n-n_2-1}$ for $i\in\Z$.
Let
\begin{equation*}
  \hat{\Gamma}_2:=\hat{\Gamma}_2(v,w)
  :=\{u\in W^{1,2}_{loc}(\R^{n_2+1}\times\T^{n-n_2-1})\,|\, v\leq u\leq w\}.
\end{equation*}
For $\bd{p}\leq \bd{q}\in\Z^{n_2}$,
set $[\bd{p},\bd{q}]:=[\bd{p}_{1},\bd{q}_1 ]\times \cdots \times [\bd{p}_{n_2},\bd{q}_{n_2}]$.
Let $\bd{1}:=(1,\cdots, 1)\in \N^{n_2}$.
For $\bd{k}\in\Z^{n_2}$, set
\begin{equation*}
  \begin{split}
    T_{\bd{k}}  :=&[\bd{k},\bd{k}+\bd{1}]\times[0,1]\times \T^{n-n_2-1} \\
     := & [\bd{k}_1, \bd{k}_1+1]\times \cdots [\bd{k}_{n_2}, \bd{k}_{n_2}+1]\times[0,1]\times \T^{n-n_2-1}.
  \end{split}
\end{equation*}
For $u\in\hat{\Gamma}_2$ and $\bd{p}\leq \bd{q}\in\Z^{n_2}$,
we have
\begin{equation}\label{eq:s}
\begin{split}
&J_{1;\bd{p},\bd{q}}(u)\\
:=&\sum^{\bd{q}}_{\bd{k}=\bd{p}}\int_{T_{\bd{k}}}[L(u)-L(v)]\ud x\\
=& \sum^{\bd{q}}_{\bd{k}=\bd{p}}\int_{T_{\bd{k}}}\Big\{\frac{1}{2}|\nabla (u-v)|^2+\nabla (u-v)\nabla v +\left[F(x,u)-F(x,v)\right]\Big\}\ud x\\
=&:I_1+I_2+I_3,
\end{split}
\end{equation}
where $\sum^{\bd{q}}_{\bd{k}=\bd{p}}=\sum^{\bd{q}_1}_{\bd{k}_1=\bd{p}_1}\cdots \sum^{\bd{q}_{n_2}}_{\bd{k}_{n_2}=\bd{p}_{n_2}}$.
First $I_3$ can be estimated as in \eqref{eq:4.5}, so
\begin{equation*}
  \lim_{\bd{p}\to -\infty\atop \bd{q}\to \infty}I_3
\end{equation*}
exists, where
$\bd{p}\to -\infty$ (resp. $\bd{q}\to \infty$) means
$\bd{p}_i\to -\infty$ (resp. $\bd{q}_i\to \infty$) for all $1\leq i\leq n_2$.

Next we have
\begin{equation}\label{eq:dfk}
\begin{aligned}
I_2:=& \sum^{\bd{q}}_{\bd{k}=\bd{p}}\int_{T_{\bd{k}}} \nabla(u-v) \cdot \nabla v \ud x\\
=& \sum^{\bd{q}}_{\bd{k}=\bd{p}}\int_{\partial T_{\bd{k}}}(u-v) 
\frac{\partial v}{\partial \nu} \ud S -\sum^{\bd{q}}_{\bd{k}=\bd{p}}\int_{T_{\bd{k}}}(u-v) \Delta v \ud x.
\end{aligned}
\end{equation}
Since the second term on the right hand of \eqref{eq:dfk} can be estimated as in the proof of Proposition \ref{prop:2.8},
we only need to calculate the first term on the right hand of \eqref{eq:dfk}.
Notice that
\begin{equation*}
\begin{split}
  &\int_{\partial T_{\bd{k}}}(u-v) \frac{\partial v}{\partial \nu} \ud S\\
  =& \sum_{i=1}^{n_2}\int_{{[\bd{k}_1,\bd{k}_1+1]\times\cdots \times\{x_i=\bd{k}_i \textrm{ or }  \bd{k}_i+1\}\times\cdots \times [\bd{k}_{n_2},\bd{k}_{n_2}+1]\times[0,1]\times \T^{n-n_2-1}}}(u-v) \frac{\partial v}{\partial \nu} \ud S\\
  &\quad +\int_{{[\bd{k},\bd{k}+\bd{1}]\times\{x_{n_2+1}=0\}\times \T^{n-n_2-1}}}(u-v) \frac{\partial v}{\partial \nu} \ud S\\
  &\quad +\int_{{[\bd{k},\bd{k}+\bd{1}]\times\{x_{n_2+1}=1\}\times \T^{n-n_2-1}}}(u-v) \frac{\partial v}{\partial \nu} \ud S\\
  =:&I_{2,1}+I_{2,2}+I_{2,3}.
  \end{split}
\end{equation*}
So only $I_{2,2}$ and $I_{2,3}$ are needed to be calculated. Since they are estimated similarly,
we give the estimation of $I_{2,2}$.
Calculate
\begin{equation}\label{eq:lf}
\begin{split}
  |\sum^{\bd{q}}_{\bd{k}=\bd{p}}I_{2,2}|\leq &\sum^{\bd{q}}_{\bd{k}=\bd{p}}
\int_{[\bd{k},\bd{k}+\bd{1}]\times\{x_{n_2+1}=0\}\times \T^{n-n_2-1}}|(u-v) \frac{\partial v}{\partial \nu}| \ud S\\
   \leq &\sum^{\bd{q}}_{\bd{k}=\bd{p}}\norm{\frac{\partial v}{\partial x_{n_{2}+1}}}_{L^{\infty}(\R^n)} \int_{[\bd{k},\bd{k}+\bd{1}]\times\{x_{n_2+1}=0\}\times \T^{n-n_2-1}}(w-v) \ud S.
\end{split}
\end{equation}
Since
\begin{equation*}
  \sum_{\bd{k}\in\Z^{n_2}}\int_{[\bd{k},\bd{k}+\bd{1}]\times\{x_{n_2+1}=0\}\times \T^{n-n_2-1}}(w-v) \ud S
  \leq \textrm{Volume of $\T^{n}$ in $\R^{n}$} = 1,
\end{equation*}
\begin{equation*}
|I_{2,2}|\leq \norm{\frac{\partial v}{\partial x_2}}_{L^{\infty}(\R^{n})}
\end{equation*}
and
\begin{equation*}
  \lim_{\bd{p}\to -\infty \atop \bd{q}\to \infty}I_{2,2}
\end{equation*}
exists.
Hence by Lemma \ref{lem:moser},
\begin{equation*}
  \lim_{\bd{p}\to -\infty \atop \bd{q}\to \infty}\sum^{\bd{q}}_{\bd{k}=\bd{p}}\int_{\partial T_{\bd{k}}} (u-v) \frac{\partial v}{\partial \nu} \ud S
\end{equation*}
exists.
Therefore we have proved that
$J_{1;\bd{p},\bd{q}}(u)$ is bounded from below and
\begin{equation}\label{eq:4.9}
\begin{split}
  J_1(u):=&\lim_{\bd{p}\to -\infty \atop \bd{q}\to \infty}J_{1;\bd{p},\bd{q}}(u)\\
  =&\frac{1}{2}\norm{\nabla(u-v)}^2 _{L^2(S_0)}+ \int_{\R^{n_2}\times \{x_{n_2+1}=0 \,\,\textrm{or}\,\, 1\} \times \T^{n-n_2-1}}(u-v) \frac{\partial v}{\partial \nu} \ud S\\
  &\quad -\int_{S_{0}}\left(u-v\right) \Delta v \ud x +\int_{S_0}[F(x,u)-F(x,v)]\ud x
  \end{split}
\end{equation}
is well-defined.
Note that \eqref{eq:4.8} also holds, i.e.,
\begin{equation}\label{eq:4.81}
  J_{1}(u)=\infty \Longleftrightarrow \norm{\nabla (u-v)}_{L^2 (S_0)}=\infty.
\end{equation}

Now let us introduce the second renormalized functional, $J_2$ on $\hat{\Gamma}_2$.
Using the notation of Rabinowitz and Stredulinsky,
for $j\in\Z$ and $k=1,\cdots, n$,
set
$\tau_{-j}^{k}u(x):=u(x+j\bd{e}^k)$ with $\bd{e}^k$ the $k$-th unit vector of the coordinate systems.
For $u\in\hat{\Gamma}_2$ and $i\in\Z$, let
\begin{equation*}
J_{2, i}(u) := J_{1}\left(\tau_{-i}^{n_2+1} u\right)=J_{1}\left(\left.u\right|_{S_i}\right)
\end{equation*}
and
\begin{equation*}
J_{2;p,q}(u) := \sum_{i=p}^{q}J_{2,i}(u)
\end{equation*}
for $p\leq q\in\Z$.

\begin{prop}\label{prop:4.10}
Suppose $u\in\hat{\Gamma}_2$ and $p\leq q\in\Z$.
Then there is a constant $K_2=K_2(v,w,\alpha) \geq 0$ such that
\begin{equation*}
    J_{2;p,q}(u)\geq -K_2.
\end{equation*}
\end{prop}
\begin{proof}
The proof is similar to that of \cite[Proposition 4.10]{RS} 
(see also Proposition \ref{prop:2.8})  
so we omit it here.
\end{proof}
For $u\in\hat{\Gamma}_2$, set
\begin{equation*}
  J_2(u)=\liminf_{p\to -\infty \atop q\to \infty}J_{2;p,q}(u).
\end{equation*}
Note that by \eqref{eq:4.81},
\begin{equation*}
\|\nabla(u-v)\|_{L^{2}\left(S_{i}\right)}=\infty \textrm{ for some } i \in \mathbb{Z} \Longrightarrow J_{2}(u)=\infty.
\end{equation*}
Similar to Lemma \ref{lem:2.22}, we have
\begin{lem}\label{lem:4.14}
If $u \in \hat{\Gamma}_{2}$, $p, q \in \mathbb{Z}$ with $p\leq q$, then
\begin{equation}\label{eq:4.15}
J_{2 ; p, q}(u) \leq J_{2}(u)+2 K_{2}.
\end{equation}
\end{lem}
Set
\begin{equation*}
\begin{split}
\Gamma_{2} := &\Gamma_{2}(v, w) \\
:=&\left\{u \in \hat{\Gamma}_{2} \,\Big|\,\lim_{i\to -\infty}\|u-v\|_{L^{2}\left(S_{i}\right)} = 0 \text { and } \right. \\
&\quad \quad\:\quad\quad \left.\lim_{i\to \infty}\|u-w\|_{L^{2}\left(S_{i}\right)} = 0 \right\}.
\end{split}
\end{equation*}

\begin{prop}\label{prop:2.24}
If $u\in\Gamma_2$ and $J_2(u)<\infty$, then
{\setlength\arraycolsep{2pt}
\begin{eqnarray}
  \lim_{|i| \rightarrow \infty}J_{2, i}(u) &= & 0,  \label{eq:2.25}\\
  \lim_{i \rightarrow -\infty}\left\|\tau_{-i}^{n_2+1} u-v\right\|_{W^{1,2}\left(S_{0}\right)} &=& 0, \label{eq:2.26}\\
  \lim_{i \rightarrow \infty}\left\|\tau_{-i}^{n_2+1} u-w\right\|_{W^{1,2}\left(S_{0}\right)} &= & 0,  \label{eq:2.27}\\
  J_{2}(u)&=&\lim _{p \rightarrow-\infty \atop q \rightarrow \infty} J_{2 ; p, q}(u) . \label{eq:2.28}
\end{eqnarray}}
\end{prop}
\begin{proof}
The proof is similar to that of \cite[Proposition 4.16]{RS}  
so we omit it here.
\end{proof}

\begin{cor}\label{cor:2.49}
Suppose $u\in\hat{\Gamma}_2(v,w)$, $J_2(u)<\infty$, and $u\leq \tau_{-1}^{n_2+1}u$.
Then either (i)  $u\in\MM_1$\footnote[1]{Recall that $\MM_1:=\MM_1(v,w):=\{\phi\in\Gamma_1(v,w) \,|\, J_1(\phi)=c_1=0\}$.},  or  (ii)  there  are $\phi,\psi\in\MM_1$ with $v\leq \phi<\psi\leq w$ such  that
$u\in\Gamma_2(\phi,\psi)$.
\end{cor}
\begin{proof}
The proof is similar to that of \cite[Corollary 2.49]{RS}  
so we omit it here.
\end{proof}

We also need to show that $J_{2,i}$ is weakly lower semicontinuous.

\begin{lem}\label{lem:4.26}
Suppose $i\in\Z$ and $\MY\subset \hat{\Gamma}_2$ with $J_{2,i}(u)<\infty$ for all $u\in\MY$.
Then $J_{2,i}$ is weakly lower semicontinuous (with respect to $\norm{\cdot}_{W^{1,2}(S_i)}$) on $\MY$.
\end{lem}
\begin{proof}
The proof is similar to that of \cite[Lemma 4.26]{RS}  
so we omit it here.
\end{proof}

Following \cite{RS} a compactness property of $J_2$ is needed.

\begin{prop}\label{prop:4.29}
Let $\MY\subset \hat{\Gamma}_2(v,w)$ with the property
\begin{enumerate}
  \item[$(Y^2 _1)$] \label{eq:Y12}
if $u\in\MY$ and $\chi_{R}\in \hat{\Gamma}_2(v,w)$ with
$\chi_R (x)=u(x)$ for $|x_{n_2+1}|\geq R$,
then $\chi_R \in\MY $ for all large $R$.
\end{enumerate}
Define
\begin{equation}\label{eq:4.30}
  c(\MY)=\inf_{u\in\MY}J_2(u).
\end{equation}
If $c(\MY)<\infty$ and $(u_k)$ is a minimizing sequence for
\eqref{eq:4.30},
then there is a $U\in\hat{\Gamma}_2(v,w)$ such that along a subsequence,
$u_k-v\to U-v$ in $W^{1,2}(S_i)$ for all $i\in\Z$.
\end{prop}

\begin{proof}
The proof is similar to that of \cite[Proposition 4.29]{RS}  
so we omit it here.
\end{proof}

We have an analogue of Proposition \ref{prop:2.64}.

\begin{prop}\label{prop:4.35}
Under the hypotheses of Proposition \ref{prop:4.29},
suppose
\begin{enumerate}
  \item[$(Y^2 _2)$] \label{eq:Y22}
there is a minimizing sequence $(u_k)$ for \eqref{eq:4.30} such that for some $r\in (0,\frac{1}{2})$,
some $z\in\R$, all smooth $\phi$ with support in $B_r(z)$, and associated $t_0(\phi)>0$,
\begin{equation}\label{eq:4.35-1}
c(\MY) \leq J_{2}\left(u_{k}+t \phi\right)+\delta_{k}
\end{equation}
for all $|t|\leq t_0(\phi)$, where $\delta_k=\delta_k(\phi)\to 0$ as $k\to \infty$.
\end{enumerate}
Then the weak limit $U$ of $u_k$ satisfies \eqref{eq:PDE} in $B_r(z)$.
\end{prop}
\begin{proof}
The proof is similar to that of Proposition \ref{prop:2.64} 
(see also \cite[Propositions 2.64 and 4.35]{RS}) so we omit it here.
\end{proof}

As in \cite{RS}, for $\tilde{v}\in\MM_1\setminus \{v,w\}$, set
\begin{equation*}
\Gamma_{2}(\tilde{v})=\left\{u \in \hat{\Gamma}_{2}\left(v, w \right) |\left\|\tau_{-i}^{n_2+1} u-\tilde{v}\right\|_{L^{2}\left(S_{0}\right)} \rightarrow 0 \text { as }|i| \rightarrow \infty\right\}.
\end{equation*}
For $\tilde{v}\in\{v,w\}$, let
\begin{equation}\label{eq:gamma1}
\begin{split}
  \Gamma_2(\tilde{v})=\{u\in W^{1,2}_{loc}(\R^{n_2+1}\times\T^{n-n_2-1})\,|\,
  &\textrm{either (a) } u \in \hat{\Gamma}_{2}\left(v, w \right) \textrm{ such that } \\
  &\left\|\tau_{-i}^{n_2+1} u-\tilde{v}\right\|_{L^{2}\left(S_{0}\right)} \rightarrow 0 \text { as }|i| \rightarrow \infty, \\
  &\textrm{ or (b) $u-\tilde{v}$ has compact support}\}.
  \end{split}
\end{equation}

\begin{rem}
Although our definition of ${\Gamma}_2(\tilde{v})$ is
different with \cite[Theorem 2.72]{RS},
it is enough for the arguments in \cite{RS}.
\end{rem}

By Remark \ref{rem:kddfff}, $J_2$ is well-defined on $\Gamma_2(\tilde{v})$ for all $\tilde{v}\in \MM_1$.
\begin{rem}\label{rem:2.70}
As in \cite{RS}, 
it is easy to see that \eqref{eq:2.25} and \eqref{eq:2.28} in Proposition \ref{prop:2.24} hold 
for $u\in \Gamma_2(\tilde{v})$ satisfying $J_2(u)<\infty$, and \eqref{eq:2.26} replaced by 
$\lim\limits_{|i|\to\infty}\|\tau^{n_2 +1}_{-i}u-v\|_{W^{1,2}(S_0)}=0$.
\end{rem}
Define
\begin{equation*}
c_{2}(\tilde{v}):=\inf _{u \in \Gamma_{2}(\tilde{v})} J_{2}(u)
\end{equation*}
and set
\begin{equation*}
\mathcal{M}_{2}(\tilde{v})=\left\{u \in \Gamma_{2}(\tilde{v}) \,|\, J_{2}(u)=c_{2}(\tilde{v})\right\}.
\end{equation*}
Then we have:

\begin{thm}\label{thm:2.72}
$c_2(\tilde{v})=0$ and $\MM_2(\tilde{v})=\{\tilde{v}\}$.
\end{thm}

\begin{proof}
  Note $J_2(\tilde{v})=0$, so $c_{2}(\tilde{v}) \leq 0$.
If $\tilde{v}\in\MM_1\setminus\{v,w\}$,
similar to \cite[Theorem 4.38]{RS} (see also \cite[Theorem 2.72]{RS}),
we have $J_{2}(u) \geq 0$ for all $u\in\Gamma_2(\tilde{v})$.
Next suppose $\tilde{v}=w$.
If $u\in\Gamma_2(w)$ satisfying (a) of \eqref{eq:gamma1},
the arguments of \cite{RS} still hold and thus $J_2(u)\geq 0$ for all $u\in\Gamma_2(\tilde{v})$.
Assume $u\in\Gamma_2(w)$ satisfying (b) of \eqref{eq:gamma1}.
If $u=w$, $J_2(u)= 0$.
If $u\neq w$,
by Lemma \ref{lem:3.621} and Remark \ref{rem:lwwwwfkfk} (replace $n_2$ by $n_2+1$),
$J_2(u)\geq 0$.
Similarly the case $\tilde{v}=v$ also implies that $J_2(u)\geq 0$ for $u\in \Gamma_2(v)$.
Thus $c_{2}(\tilde{v})=0$.

Next to show 
$\MM_2(\tilde{v})=\{\tilde{v}\}$, again firstly suppose $\tilde{v}\in\MM_1\setminus\{v,w\}$.
Let $u\in\MM_2(\tilde{v})$.
Then $v\leq  u\leq w $.
\begin{itemize}
  \item If $z\in\R^n$ such that $v(z)<  u(z)< w(z) $,
choose $r\in (0, 1/2)$, such that for any $\phi$ smooth with support in $B_r(z)$,
and $|t|$ small (depending on $\phi$) satisfying $v\leq u+t\phi\leq w$.
Hence $J_2 (u+t\phi)\geq c_2(\tilde{v})=0$.
  \item If $z\in\R^n$ such that $u(z)=v(z)$ or $u(z)= w(z)$.
For any fixed $r\in (0, 1/2)$, $\phi$ smooth with support in $B_r(z)$,
and $|t|$ small (depending on $\phi$),
\begin{equation}\label{eq:dgklfd223}
J_2 (u+t\phi)
\geq
J_2 \left( \max\left(\min(u+t\phi, w),v\right) \right).
\end{equation}
Indeed,
\begin{equation*}
  \begin{split}
  &J_2\left(\max\left(\min(u +t\phi, w),v\right)\right)+
  J_2\left(\min\left(\min(u +t\phi, w),v\right)\right)\\
  =&J_2(\min(u +t\phi, w))+J_2(v)\\
  =&J_2(\min(u +t\phi, w)),
  \end{split}
\end{equation*}
and
\begin{equation*}
\begin{split}
&J_2(\max(u +t\phi, w))+J_2(\min(u +t\phi, w))\\
=&J_2(u +t\phi)+J_2(w)\\
=&J_2(u +t\phi).
 \end{split}
 \end{equation*}
Since $\min\left(\min(u +t\phi, w),v\right)$ and $\max(u +t\phi, w)$ satisfy
(b) of \eqref{eq:gamma1} for $\Gamma_2(v)$ and $\Gamma_2(w)$, respectively,
by $c_2(w)=c_2(v)=0$,
\begin{equation*}
  J_2\left(\min\left(\min(u +t\phi, w),v\right)\right) \geq 0,
   \quad J_2(\max(u +t\phi, w)) \geq 0.
\end{equation*}
Since $\max(\min(u+t\phi, w),v )\in \hat{\Gamma}_2(v,w)$,
\eqref{eq:dgklfd223} holds and
$J_2 (u+t\phi)\geq c_2(\tilde{v})=0$.
\end{itemize}
Hence for $u_k=u$,
note that \eqref{eq:4.35-1} of Proposition \ref{prop:4.35} (with $\delta_k=0$) is satisfied.
Consequently, $u$ satisfies \eqref{eq:PDE} for all $z\in\R^n$.
Repeating the proof of \cite[Theorem 2.72]{RS} (see also \cite[Theorem 4.38]{RS}),
the proof of Theorem \ref{thm:2.72} is complete for the case $\tilde{v}\in\MM_1\setminus\{v,w\}$.

Assume $\tilde{v}=w$.
If $u\in\MM_2(w)$ satisfying (a) of \eqref{eq:gamma1},
$v\leq u\leq w$.
Similar to \eqref{eq:dgklfd223},
$0\leq J_2(u+t\phi)$.
So setting $u_k=u$ implies $u$ is a solution of \eqref{eq:PDE}.
Proceeding as for the case of $\tilde{v}\in\MM_1\setminus\{v,w\}$ shows $u=w$.
If $u(\neq w)\in\MM_2(w)$ satisfying (b) of
\eqref{eq:gamma1}, we claim $J_2(u)>0$.
Indeed,
If $J_2(u)=0$,
then for any $z\in\R^n$,
for some $r\in (0,\frac{1}{2})$, all smooth $\phi$ with support in $B_r(z)$,
and associated $t_0(\phi)>0$,
$u+t\phi$ satisfies (b) of \eqref{eq:gamma1}.
Since $c_2(w)=0$,
$J_2(u+t\phi)\geq 0$.
Setting $u_k\equiv u$ implies $u$ is a solution of \eqref{eq:PDE}.
Since $u\neq w$, 
the argument of showing $\MM_1 ^{\bd{l}}$ is an ordered set in 
the proof of Proposition \ref{prop:2.2}
can be applied to show
that $u<w$ or $u>w$.
Both cases contradict $u-w$ has compact support.
So $J_2(u)>0$ and
thus $\MM_2(w)=\{w\}$.
The case of $\tilde{v}=v$ can be proved similarly.
\end{proof}


\subsection{Heteroclinic solution for the case $rank(\bar{\Lambda}_3)=rank(\bar{\Lambda}_2)-1$}\label{sec:4.4}
\

Assume
\begin{equation}\label{eq:*0}
\textrm{$v_1,w_1\in \mathfrak{M}(\bar{a}^1)$ are an adjacent pair}, \tag{$*_1$}
\end{equation}
i.e.,
there does not exist $u\in\mathfrak{M}(\bar{a}^1)$ such that $v_1\leq u\leq w_1$.
%
Then for all $i\in\Z$,
$\tau_{i}^{1}v_1, \tau_{i}^{1}w_1\in\mathfrak{M}(\bar{a}^1)$
are also adjacent.
Thus without loss of generality,
suppose $\tau_{-1}^{1}w_1>\tau_{-1}^{1}v_1\geq w_1>v_1 \geq \tau_{1}^{1}w_1>\tau_{1}^{1}v_1$.

With the above preliminaries,
we obtain the main existence result as in \cite[Theorems 3.2, 4.40]{RS}.
To formulate it,
set $\Gamma_{2}:=\Gamma_{2}\left(v_{1}, w_{1}\right)$ and
\begin{equation}\label{eq:3.1}
c_{2}:=c_{2}\left(v_{1}, w_{1}\right):=\inf _{u \in \Gamma_{2}\left(v_{1}, w_{1}\right)} J_{2}(u).
\end{equation}

\begin{thm}\label{thm:3.2}
If $F$ satisfies \eqref{eq:F1-F2} and \eqref{eq:*0} holds,
then
\begin{enumerate}
  \item \label{thm:3.2-1} There  is  a $U\in \Gamma_2$ such  that $J_2(U)=c_2$,
  i.e., $\MM_2:= \MM_2(v_1,w_1):=\{u\in\Gamma_2(v_1,w_1)\,|\, J_2(u)=c_2\}\neq \emptyset$.
  \item \label{thm:3.2-2} Any $U\in\MM_2$ satisfies
  \begin{enumerate}
    \item \label{thm:3.2-2-a} $U$ is a solution of \eqref{eq:PDE}.
    \item \label{thm:3.2-2-b} $\norm{U-v_1}_{C^2(S_i)}\to 0,  \quad i\to -\infty,$\\
    $\norm{U-w_1}_{C^2(S_i)}\to 0,  \quad i\to \infty,$\\
    i.e., $U$ is heteroclinic in $x_{n_2+1}$ from $v_1$ to $w_1$,
    \item $v_1<U<\tau_{-1}^{n_2+1}U<w_1$, and $U<\MT_{\bar{\bd{k}}}U$
    for $\bar{\bd{k}}\in \bar{\Lambda}_1$ with $\bar{\bd{k}}\cdot \bar{a}^1>0$. \label{thm:3.2-2-c}
  \end{enumerate}
  \item \label{thm:3.2-3} $\MM_2$ is an ordered set.
\end{enumerate}
\end{thm}

\begin{proof}
After endowed the notations in \cite{RS} with new meanings,
the proof of Theorem \ref{thm:3.2} can be 
obtained by reapting the proof of \cite[Theorem 4.40]{RS} (see also \cite[Theorem 3.2]{RS}),
so we omit it here.
Noting by the proof of Proposition \ref{prop:kdk},
we have 
$U<\MT_{\bar{\bd{k}}}U$
for $\bar{\bd{k}}\in \bar{\Lambda}_1$ with $\bar{\bd{k}}\cdot \bar{a}^1>0$.
\end{proof}

\begin{rem}\label{rem:3.31}
Similar to \cite{RS},
we have 
$$\MM_2(w_1, v_1)=\{u\in \Gamma_2(w_1,v_1)\,|\,J_2(u)=c_2(w_1,v_1)\} \neq \emptyset$$
provided gap condition \eqref{eq:*0}.
The minimum $c_2=c_2(v_1,w_1)$ can be characterized by
$$c_2= \inf_{u\in \mathcal{S}} J_2(u),$$
where 
$$\mathcal{S}=\{u\in \hat{\Gamma}_2(v_1,w_1)\,|\,u\leq \tau_{-1}^{n_2+1}u \textrm{ and $v_1\neq u\neq w_1$}\}.$$
See \cite[Corollary 3.32]{RS} for the proof.
\end{rem}

As in \cite{RS},
the next theorem shows the gap condition \eqref{eq:*0} 
is both necessary and sufficient in order that $\MM_2\neq \emptyset$.
\begin{thm}\label{thm:3.34}
Suppose $F$ satisfies \eqref{eq:F1-F2}, and
$\bar{v}<\bar{w}\in \MMM^{rec}(\bar{a}^1)$ are adjacent.
Assume $v,w\in\MM_1:=\MM_1(\bar{v},\bar{w})$ with $v\neq w$.
Then
$\MM_2(v,w)\neq \emptyset$  if and only if $v$ and $w$ are adjacent members of $\MM_1$.
\end{thm}
\begin{proof}
  The proof is similar to that of \cite[Theorem 3.34]{RS}  
  so we omit it here.
\end{proof}

An interesting question is how the gap condition \eqref{eq:*0} depends on $F$.
For rationally rotation vector, \eqref{eq:*0} continuously depends on $F$ (cf. \cite{RS}),
and if \eqref{eq:*0} does not hold, one can perturb $F$ to regain \eqref{eq:*0}.
But for irrationally rotation vector, things are quite different.
For instance, if the rotation vector $\alpha$ is Diophantine,
a theorem of Moser \cite{moser} which generalizes the famous KAM-Theory
tells us that small perturbation can not break up the foliation.
If the perturbation is large enough, Bangert \cite{bangert1987} showed that
the gap condition can be regained.
See \cite{bangert1987} for more discussion on this topic.

In the rest of this subsection,
we explore the relationship between $\MM_2(v_1,w_1)$ and $\MMM(\bar{a}^1, -\bar{\bd{e}}^{n_2+1})$.
The following lemma of Bangert will be used.
\begin{lem}[{cf. \cite[Proposition 4.2]{Bangert}}]\label{lem:bangertdkd}
Suppose $u\in\MMM(\bar{a}^1,\cdots, \bar{a}^t)$ and $t>1$.
Then there exist $u^{-}$ and $u^{+}$ in $\MMM(\bar{a}^1,\cdots, \bar{a}^{t-1})$ with
the following properties:
\begin{enumerate}
  \item If $\bar{\bd{k}}^i\in\bar{\Lambda}_t$ and $\lim_{i\to\infty}\bar{\bd{k}}^i\cdot \bar{a}^t=\infty$ (resp. $\lim_{i\to\infty}\bar{\bd{k}}^i\cdot \bar{a}^t=-\infty$), then $\lim_{i\to\infty}\MT_{\bar{\bd{k}}^i}u=u^{+}$ (resp. $\lim_{i\to\infty}\MT_{\bar{\bd{k}}^i}u=u^{-}$).
  \item $u^-<u<u^+$ and $\MT_{\bar{\bd{k}}}u^- \geq u^+$ if $\bar{\bd{k}}\in\bar{\Lambda}_s$ and $\bar{\bd{k}}\cdot \bar{a}^s>0$ for some $1\leq s<t$.
  \item If $\lim_{i\to\infty}\MT_{\bar{\bd{k}}^i}u=v$ for a sequence $\bar{\bd{k}}^i\in\bar{\Lambda}_t$ such that $\bar{\bd{k}}^i\cdot \bar{a}^t$ is bounded then $v\in\MMM(\bar{a}^1,\cdots, \bar{a}^t)$.
\end{enumerate}
\end{lem}

\begin{rem}\label{rem:dfdd}
Recall that the limit is taken with respect to $C^1$-convergence on compact sets.
For any $\ell\in\Z$,
$\norm{\MT_{\bar{\bd{k}}^i}u-u^{+}}_{L^2(S_{\ell})}\to 0$
(resp. $\norm{\MT_{\bar{\bd{k}}^i}u-u^{-}}_{L^2(S_{\ell})}\to 0$)
as
$\bar{\bd{k}}^i\in\bar{\Lambda}_t$ and
$\bar{\bd{k}}^i\cdot \bar{a}^t\to \infty$
(resp. $\bar{\bd{k}}^i\cdot \bar{a}^t\to -\infty$).
\end{rem}

\begin{thm}\label{thm:kdfdfdf}
Let F satisfy \eqref{eq:F1-F2}
and assume \eqref{eq:*0} holds.
Suppose
$$\alpha=(\alpha_1,\cdots, \alpha_{n_2},0,\cdots,0)\in\R^n \setminus\Q^n$$
with $\alpha_i\not\in\Q$ for $1\leq i\leq n_2$ and $\alpha_1,\cdots, \alpha_{n_2}$ are linearly independent.
Then
$$\MMM(\bar{a}^1, -\bar{\bd{e}}^{n_2+1})\cap \Gamma_2(v_1,w_1)=\MM_2(v_1,w_1).$$
\end{thm}

\begin{proof}
For $u\in\MM_2(v_1,w_1)$,
by \eqref{thm:3.2-2-c} of Theorem \ref{thm:3.2} and the periodicity of $u$,
$u$ is WSI.
Proceeding as in the proof of Proposition \ref{prop:ped} 
shows that $u$ is minimal.
By Lemma \ref{prop:dk68956},
$u\in\MMM(\bar{a}^1, -\bar{\bd{e}}^{n_2+1}) $.

Next suppose $u\in\MMM(\bar{a}^1, -\bar{\bd{e}}^{n_2+1}) $.
By Lemma \ref{lem:bangertdkd},
there exist $v$, $w\in \MMM(\bar{a}^1)$ with $v<w$ satisfying:
if $\bar{\bd{k}}\in\bar{\Lambda}_1$ and $\bar{\bd{k}}\cdot \bar{a}^1>0$,
then
$\MT_{\bar{\bd{k}}}v \geq w$.
So $\MM_1$, $J_1$, $J_2$, etc. are well-defined.
By Remark \ref{rem:dfdd},
$u \in \Gamma_{2}(v, w)$.
Next it will be shown that
\begin{equation}\label{eq:3.76}
J_{2}(u)=c_{2}(v, w)=:c_2,
\end{equation}
so
$u\in\MM_2(v,w)$.
But \eqref{eq:3.76} can be proved as \cite[(3.76)]{RS}.
By Theorem \ref{thm:3.34},
$v, w$ are adjacent in $\MMM(\bar{a}^1)$.
Thus by \eqref{eq:*0},
$v=v_1, w=w_1$.
Hence $u\in \MM_2(v_1,w_1)$,
i.e.,
$\MMM(\bar{a}^1, -\bar{\bd{e}}^{n_2+1})\cap\hat{\Gamma}_2(v_1,w_1)\subset \MM_2(v_1,w_1)$.
\end{proof}

\begin{cor}\label{cor:dfdfd}
Suppose
$$\alpha=(\alpha_1,\cdots, \alpha_{n_2},0,\cdots,0)\in\R^n \setminus\Q^n$$
with $\alpha_i\not\in\Q$ for $1\leq i\leq n_2$ and $\alpha_1,\cdots, \alpha_{n_2}$ are linearly independent.
Then
$$\MMM(\bar{a}^1, -\bar{\bd{e}}^{n_2+1})=\cup_{\langle v_1,w_1 \rangle\in \mathcal{A}_1}\MM_2(v_1,w_1)$$
with
$$\mathcal{A}_1=\{\langle v_1,w_1 \rangle\,|\, v_1<w_1\in \MMM(\bar{a}^1) \textrm{ are an adjacent pair} \}.$$
\end{cor}

\begin{proof}
By Theorem \ref{thm:kdfdfdf},
$$\MMM(\bar{a}^1, -\bar{\bd{e}}^{n_2+1})\supset\cup_{\langle v_1,w_1 \rangle\in \mathcal{A}_1}\MM_2(v_1,w_1).$$

For any $u\in \MMM(\bar{a}^1, -\bar{\bd{e}}^{n_2+1})$,
by the proof of Theorem \ref{thm:kdfdfdf},
$u\in\MM_2(u^-,u^+)$,
where
$u^-<u^+$ are given by Lemma \ref{lem:bangertdkd},
and $u^-,u^+$ are adjacent in $\MMM(\bar{a}^1)$.
This proves $\MMM(\bar{a}^1, -\bar{\bd{e}}^{n_2+1})\subset \cup_{\langle v_1,w_1 \rangle\in \mathcal{A}_1}\MM_2(v_1,w_1).$
\end{proof}


\subsection{The second renormalized  functional for the case $1\leq rank(\bar{\Lambda}_3)\leq rank(\bar{\Lambda}_2)-2$}\label{sec:4.3}
\

The case $1\leq rank(\bar{\Lambda}_3)\leq rank(\bar{\Lambda}_2)-2$ 
will be treated as in Section \ref{sec:3.2}.
First we construct a minimal and WSI solution $u\in\MMM(\bar{a}^1,\bar{a}^2)$.
The proof is a combination of the methods of Moser \cite{moser}, Bangert \cite{Bangert} and Rabinowitz--Stredulinsky \cite{RS}.

\begin{thm}\label{thm:dklkkff}
Suppose \eqref{eq:*0} holds. Then
$\MMM(\bar{a}^1,\bar{a}^2)\neq \emptyset.$
\end{thm}

\begin{proof}
Choose $\bar{a}^2(k)$ such that (cf. \cite[the proof of (6.5)]{moser})
\begin{enumerate}
  \item $(\bar{a}^1,\bar{a}^2(k))$ is admissible for all $k\in\N$;
  \item $rank(\bar{\Lambda}_3 (k))=rank(\bar{\Lambda}_2) -1$ with $\bar{\Lambda}_3(k):=\Z^{n+1}\cap \langle \bar{a}^1, \bar{a}^2(k)\rangle ^{\perp}$;
  \item  $\bar{\Lambda}_3\subset \bar{\Lambda}_3(k)$; \label{eq:dllgdf}
  \item $\bar{a}^2(k)\to \bar{a}^2$ as $k\to \infty$.
\end{enumerate}
By Theorem \ref{thm:3.2}, there are $u_k\in\MMM(\bar{a}^1,\bar{a}^2(k))$ such that $v_1\leq u_k \leq w_1$.
Similar to \cite[(3.5) or (4.41)]{RS},
we can normalize $u_k$ as follows.
Without loss of generality,
assume $rank(\Lambda_3)=n-n_2-n_3$,
and $\bd{e}^1, \cdots, \bd{e}^{n_2}$ are perpendicular to $\Lambda_2$,
and $\bd{e}^{n_2+1}, \cdots, \bd{e}^{n_2+n_3}$ are perpendicular to $\Lambda_3$ in the subspace $V_2$.
Thus by the asymptotic behavior of $u_k$,
\begin{equation}\label{eq:3.511}
\begin{split}
\int_{[0,1]^{n_2}\times [-1,0]^{n_3}\times \T^{n-n_2-n_3}} u_k \ud x
\leq &\frac{1}{2} \int_{[0,1]^{n_2+n_3}\times \T^{n-n_2-n_3}}\left(v_{1}+w_{1}\right) \ud x\\
\leq &\int_{[0,1]^{n_2+n_3}\times \T^{n-n_2-n_3}} u_k \ud x.
\end{split}
\end{equation}
By Lemma \ref{lem:dldqwloplo},
there is a minimal and WSI solution $u$ such that
$u_k\to u$
with respect to $C^1$-convergence on compact sets
as $k\to \infty$.
Thus
\begin{equation}\label{eq:3.52211}
\begin{split}
\int_{[0,1]^{n_2}\times [-1,0]^{n_3}\times \T^{n-n_2-n_3}} u \ud x
\leq &\frac{1}{2} \int_{[0,1]^{n_2+n_3}\times \T^{n-n_2-n_3}}\left(v_{1}+w_{1}\right) \ud x\\
\leq &\int_{[0,1]^{n_2+n_3}\times \T^{n-n_2-n_3}} u \ud x.
\end{split}
\end{equation}
So $u\neq v_1$ and $w_1$.

Since $\bar{a}^1 (u_k)\equiv \bar{a}^1$, by Lemma \ref{lem:dldqwloplo},
$\bar{a}^1 (u)= \bar{a}^1$.
By \eqref{eq:dllgdf}, $u$ satisfies $\MT_{\bar{\bd{k}}}u=u$ for $\bar{\bd{k}}\in\bar{\Lambda}_3$.
We claim:
\begin{equation}\label{eq:dlldldd}
  \textrm{$t(u)=2$ and $\bar{a}^2 (u)= \bar{a}^2$.}
\end{equation}
If $t(u)=1$,
$v_1\leq u \leq w_1$, since $v_1\leq u_k \leq w_1$.
Note $v_1$ and $w_1$ are adjacent, then $u=v_1$ or $w_1$, a contradiction.
Thus $t(u)=2$.
To prove \eqref{eq:dlldldd},
we use Bangert's argument (\cite[the proof of Lemma 3.10]{Bangert}).
If $\bar{a}^2 (u)\neq \bar{a}^2$,
choose $\bar{\bd{k}}\in\Z^{n+1}\setminus \{0\}$
such that
$\bar{\bd{k}}\cdot\bar{a}^2(u)>0$ and $\bar{\bd{k}}\cdot\bar{a}^2(u_k)<0$ for large $k$.
Then
$\MT_{\bar{\bd{k}}}u>u$ and $\MT_{\bar{\bd{k}}}u_k< u_k$ for large $k\in\N$.
This contradicts $u_k\to u$ on compact set.
Thus \eqref{eq:dlldldd} is proved and
$u\in\MMM(\bar{a}^1,\bar{a}^2)$.
\end{proof}

With Theorem \ref{thm:dklkkff} in hand,
proceeding as in Section \ref{sec:3.2},
by Lemma \ref{lem:659712},
either
$\MMM^{rec}(\bar{a}^1,\bar{a}^2)$ foliates the gap between $v_1$ and $w_1$,
or
$\MMM^{rec}(\bar{a}^1,\bar{a}^2)$ constitutes a lamination of $(v_1, w_1):=\{(x,x_{n+1}) \,|\, v_1(x)<x_{n+1}<w_1(x)\}$.
In the latter case,
$$\{u(\bd{0})\,|\, u\in \MMM^{rec}(\bar{a}^1,\bar{a}^2)\cap \Gamma_2(v_1, w_1)\}(\subset \{u(\bd{0})\,|\, u\in \MMM(\bar{a}^1,\bar{a}^2)\cap \Gamma_2(v_1, w_1)\})$$ is a Cantor set.
If $\{u(\bd{0})\,|\, u\in \MMM^{rec}(\bar{a}^1,\bar{a}^2)\cap \Gamma_2(v_1, w_1)\}$ is a Cantor set,
we can construct the functional $J_1$ as in Section \ref{sec:3}
and obtain the variational construction of $u\in \MMM(\bar{a}^1,\bar{a}^2)\setminus \MMM^{rec}(\bar{a}^1,\bar{a}^2)$.
That is (similar to Theorem \ref{thm:ldfjldf}):
\begin{thm}\label{thm:dlfjldfjldgf}
 $\MMM(\bar{a}^1,\bar{a}^2)=\cup_{\langle v,w\rangle \in  \mathcal{A}_2} \MM_1(v,w) \cup \MMM^{rec}(\bar{a}^1,\bar{a}^2)$,
with
\begin{equation*}
\mathcal{A}_2:=\{\langle v,w\rangle\,|\,
  \textrm{$v<w\in\MMM^{rec}(\bar{a}^1,\bar{a}^2)$ are adjacent}\}.
\end{equation*}
Note here $\MM_1(v,w)$ is corresponding to the new functional $J_1$.
\end{thm}


\subsection{The second renormalized  functional for the case $rank(\bar{\Lambda}_3)=0$}
\

Proceeding as in Section \ref{sec:4.3}
and using the same idea of
Section \ref{sec:3.3},
we obtain the second renormalized  functional for the case $rank(\bar{\Lambda}_3)=0$.

\bigskip

For higher dimensions,
if more gap conditions are given,
we have more basic heteroclinic solutions.
Moreover, more complicated heteroclinic and homoclinic solutions can be constructed as in
\cite{RS, Rabi2006, Rabi2007, Rabi2011, Rabi2014}.
The details for establishing these solutions will not be given anymore.


\subsection*{Acknowledgments}
The author is greatly indebted to
Professor Zhi-Qiang Wang (Utah State University) 
for his active interest, valuable encouragements and helpful discussions 
during the preparation of the paper. 
The author is supported by NSFC: 12201162.


\end{CJK*}
\end{document}